\documentclass[reqno,11pt]{amsart}
\usepackage{mathrsfs}
\usepackage[T1]{fontenc}
\usepackage{mathrsfs}
\usepackage{latexsym}
\usepackage[dvips]{graphics}
\usepackage{epsfig}
\usepackage{amsmath,amsfonts,amsthm,amssymb,amscd}
\input amssym.def
\input amssym.tex
\usepackage{color}
\usepackage[all]{xypic}

\newcommand{\me}{\mathrm{e}}

\newcommand{\mpi}{\mathrm{\pi}}

\newtheorem{remark}{Remark}[section]
\newtheorem{lemma}{Lemma}[section]

\newtheorem{prop}{Proposition}[section]
\newtheorem{theorem}{Theorem}[section]

\begin{document}

\title{On the blow-up structure for the generalized periodic Camassa-Holm and Degasperis-Procesi equations}
\author{Ying Fu}
\address{Ying Fu\newline
Department of Mathematics\\
Northwest University\\
Xi'an, 710069\\
P. R. China}
\email{fuying@nwu.edu.cn}
\author{Yue Liu}
\address{Yue Liu (corresponding author)\newline
 Department of Mathematics, University of Texas, Arlington, TX 76019-0408}
\email{yliu@uta.edu}

\author{Changzheng Qu}
\address{Changzheng Qu\newline
Department of Mathematics\\
Northwest University\\
Xi'an, 710069\\
P. R. China} \email{czqu@nwu.edu.cn}
\thanks{The work of Fu is partially supported by the NSF-China grant-11001219 and the Research Foundation of Northwest University in China grant-09NW23.
The work of Liu is partially supported by the NSF grant
DMS-0906099 and the NHARP  grant-003599-0001-2009. The work of Qu is supported in part by the NSF-China for Distinguished Young Scholars grant-10925104.}

\maketitle \numberwithin{equation}{section}

\begin{abstract}  Considered herein are the  generalized
Camassa-Holm and Degasperis-Procesi equations in the spatially
periodic setting. The precise blow-up scenarios of strong solutions
are derived for both of equations. Several conditions on the initial
data guaranteeing the development of singularities in finite time
for strong solutions of these two equations are established. The
exact blow-up rates are also determined. Finally, geometric
descriptions of these two integrable equations from non-stretching
invariant curve flows in centro-equiaffine geometries,
pseudo-spherical surfaces and affine surfaces are given.
\end{abstract}
\vskip 0.1cm


\noindent \small 2000 {\it Mathematics Subject Classification.} \
{\it Primary:} 35B30, 35G25.

\section{Introduction}
In this paper, we are concerned with the initial-value problem
associated with the generalized periodic Camassa-Holm ($\mu$CH)
equation \cite{khe}
\begin{equation}\label{e1.1}
\left\{
 \begin{array}{ll}
\begin{split}
&\mu(u_t)-u_{xxt}=-2\mu(u)u_x+2u_xu_{xx}+uu_{xxx}, \quad t > 0,  \quad x \in \mathbb{R}, \\
&u(0,x)=u_0(x), \qquad x \in \mathbb{R}, \\
&u(t,x+1)=u(t,x),\quad\quad\quad t \ge 0, \quad x\in\mathbb{R},
 \end{split}
\end{array} \right.
\end{equation}
where $u(t,x)$ is a time-dependent function on the unit circle
$\mathbb{S}=\mathbb{R}/\mathbb{Z}$ and
$\mu(u)=\int_{\mathbb{S}}u(t,x)dx$ denotes its mean. Obviously, if
$\mu(u)=0$, which implies that $ \mu(u_t) = 0, $  then this equation
reduces to  the Hunter-Saxton (HS) equation \cite{hun1}, which is a
short wave limit of the Camassa-Holm (CH) equation \cite{cam}.

We also consider in the paper the initial-value problem associated with
the generalized periodic Degasperis-Procesi ($\mu$DP) equation
\cite{len2}
\begin{equation}\label{e1.2}
\left\{
 \begin{array}{ll}
\begin{split}
&\mu(u_t)-u_{xxt}=-3\mu(u)u_x+3u_xu_{xx}+uu_{xxx}, \quad t > 0, \quad x \in \mathbb{R}, \\
&u(0,x)=u_0(x),  \qquad x \in \mathbb{R}, \\
&u(t,x+1)=u(t,x),\quad\quad\quad t \ge 0, \quad x\in\mathbb{R},
 \end{split}
\end{array} \right.
\end{equation}
where $u(t,x)$ and $\mu(u)$ are the same as in the above. Setting
$\mu(u)=0$,  this equation becomes the short wave limit of the
Degasperis-Procesi (DP) equation \cite{deg1} or the $\mu$Burgers equation
\cite{len2}.

It is known that the Camassa-Holm equation and the
Degasperis-Procesi equation are the cases $\lambda=2$ and
$\lambda=3$, respectively, of the following family of equations
\begin{equation}\label{f1.1}
m_t + u m_x + \lambda u_x m = 0,
\end{equation}
with $ m = Au $ and $ A = 1 - \partial_x^2, $ where each equation in
the family admits peakons \cite{deg2} although only $\lambda=2$ and
$\lambda=3$ are believed to be integrable \cite{cam, deg1}.

It is observed that the corresponding $\mu$-version of the family
 is  also given by (\ref{f1.1}) with $ m = Au, $ $ A = \mu - \partial_x^2, $ where the choices $\lambda=2$ and $\lambda=3$
yield the generalized equations, i.e. the $\mu$CH and $\mu$DP
equations, respectively.

It is clear that the closest relatives of the $\mu$CH equation are the Camassa-Holm equation
with $ A = 1 - \partial_x^2 $
\begin{equation}\label{1.2}
u_t-u_{txx}+3uu_x=2u_xu_{xx}+uu_{xxx},
\end{equation}
and the Hunter-Saxton equation with $ A = - \partial_x^2 $
 \begin{equation}\label{1.3}
-u_{txx}=2u_xu_{xx}+uu_{xxx}.
\end{equation}
Both of the CH equation and the HS equation have attracted a lot of attention among the
integrable systems and the PDE communities. The Camassa-Holm
equation was introduced  in \cite{cam} as a shallow water
approximation and has a bi-Hamiltonian structure \cite{fuc}. The
Hunter-Saxton equation firstly appeared in \cite{hun1} as an
asymptotic equation for rotators in liquid crystals.

The Camassa-Holm equation is a completely integrable system with a
bi-Hamiltonian structure and hence it possesses an infinite sequence
of conservation laws, see \cite{con6} for the periodic case. It
admits soliton-like solutions (called peakons) in both periodic and
non-periodic setting \cite{cam}. The Camassa-Holm equation describes
geodesic flows on the infinite dimensional group ${\mathcal
D}^s(\mathbb{S})$ of orientation-preserving diffeomorphisms of the
unit circle $\mathbb{S}$ of Sobolev class $H^s$ and endowed with a
right-invariant metric by the $H^1$ inner product \cite{kou,mis2}.
The Hunter-Saxton equation also describes the geodesic flow on the
homogeneous space of the group $\mathcal{D}^s(\mathbb{S})$ modulo
the subgroup of rigid rotations $Rot(\mathbb{S})\simeq\mathbb{S}$
equipped with the $\dot{H}^1$ right-invariant metric \cite{len1},
which at the identity is
$$\langle u,v\rangle_{\dot{H}^1}=\int_{\mathbb{S}}u_xv_x dx.$$
This equation possesses a bi-Hamiltonian structure and is formally integrable (see \cite{hun2}).

Another remarkable property of the Camassa-Holm equation is the
presence of breaking waves (i.e. the solution remains bounded while
its slope becomes unbounded in finite time \cite
{cam, con1, con2, con4, con6,  mis1, wh}).
 Wave breaking is one of the most intriguing long-standing problems of water wave
theory \cite{wh}.

Another important integrable equation admitting peakon solitons is
the Degasperis-Procesi equation \cite{deg1} and it takes the form
\begin{equation*}
u_t-u_{xxt}+4uu_x=3u_xu_{xx}+uu_{xxx}.
\end{equation*}
It is regarded as a model for nonlinear shallow water dynamics and
its asymptotic accuracy is the same as for the Camassa-Holm shallow
water equation, and it can also be obtained from the governing
equations for water waves \cite{conl2}.  The Degasperis-Procesi
equation is a geodesic flow of a rigid invariant symmetric linear
connection on the diffeomorphism group of the circle \cite{esc1}.
More interestingly, it has the shock peakons in both periodic
\cite{esc3} and non-periodic setting \cite{lu}. Wave breaking
phenomena and global existence of solutions of the
Degasperis-Procesi equation were investigated in \cite{coc,
esc3,esc4,liu}, for example.

The $\mu$CH ($\lambda=2$) was introduced by Khesin, Lenells and
Misiolek \cite{khe} (also called $\mu$HS equation). Similar to the
HS equation \cite{hun1}, the $\mu$CH equation describes the
propagation of weakly nonlinear orientation waves in a massive
nematic liquid crystal  with external magnetic filed and
self-interaction. Here, the solution $ u(t, x) $ of the $\mu$CH
equation is the director field of a nematic liquid crystal, $ x $ is
a space variable in a reference frame moving with the linearized
wave velocity, and $ t $ is a slow time variable. Nematic liquid
crystals are fields consisting of long rigid molecules.  The $\mu$CH
equation is an Euler equation on $\mathcal{D}^s(\mathbb{S})$ (the
set of circle diffeomorphism of the Sobolev class $ H^s$) and it
describes the geodesic flow on $\mathcal{D}^s(\mathbb{S})$  with the
right-invariant metric given at the identity by the inner product
\cite{khe} $$\langle
u,v\rangle=\mu(u)\mu(v)+\int_{\mathbb{S}}u_xv_x d x.$$

It was shown in \cite{khe} that the $\mu$CH equation is formally
integrable and can be viewed as the  compatibility condition
 between
 $$\psi_{xx}=\xi m \psi$$
 and
 $$\psi_t=\left (\frac1{2\xi}-u \right )\psi_x+\frac1{2}u_x\psi$$
 where $\xi \in\mathbb{C} $ is a spectral parameter and $m=\mu(u)-u_{xx}$.

On the other hand, the $\mu$CH equation admits bi-Hamiltonian structure and infinite
hierarchy of conservation laws. The first few conservation laws in the hierarchy are
\begin{eqnarray*}
H_0=\int_{\Bbb S}m\ d x, \quad H_1=\frac1{2}\int_{\Bbb S} mu \ d x,
\quad H_2=\int_{\Bbb S}\left (\mu(u)u^2+\frac1{2}uu_x^2 \right )  d x.
\end{eqnarray*}
Whereas the Hunter-Saxton equation does not have any bounded traveling-wave solutions at all, the $\mu$CH equation
admits traveling waves that can be regarded as the appropriate candidates for solitons. Moreover, the $\mu$CH equation
admits not only periodic one-peakon solution $ u(t, x) = \varphi(x-ct) $ where
$$
\varphi(x) = \frac{c}{26} (12 x^2 + 23)
$$
for $ x \in [ - \frac{1}{2}, \frac{1}{2} ] $ and $ \varphi $ is extended periodically
to the real line, but also the multi-peakons of the form
$$u=\sum^N_{i=1}p_i(t)g(x-q_i(t)),$$
where $g(x)=\frac1{2}x(x-1)+\frac{13}{12}$ is the Green's function of the operator $(\mu -\partial_x^2)^{-1}$.

The $\mu$DP equation  ($\lambda = 3$) was firstly introduced by
Lenells, Misiolek and Ti\u{g}lay in \cite{len2}. It can be formally
described as an evolution equation on the space of tensor densities
over the Lie algebra of smooth vector fields on the circle ${\mathbb
S}$. As mentioned in \cite{len2}, such geometric interpretation is
not completely satisfactory. Recently, Escher, Kohlmann and Kolev
\cite{esc2} verified that the periodic $\mu$DP equation describes
the geodesic flow of a right-invariant affine connection on the
Fr\'{e}chet Lie group ${\rm Diff}^{\infty}({\Bbb S})$ of all smooth
and orientation-preserving diffeomorphisms of the circle ${\Bbb S}$.
The $\mu$DP equation admits the Lax pair formulations
$$\psi_{xxx}=-\xi m \psi$$
and
$$\psi_t=-\frac1{\xi}\psi_{xx}-u\psi_x+u_x\psi,$$
where $\xi\in\mathbb{C}$ is a spectral parameter and
$m=\mu(u)-u_{xx}$. Similar to the $\mu$CH equation, the $\mu$DP equation also admits bi-Hamiltonian
structure and infinite hierarchy of conservation laws, and it is formally integrable \cite{len2}. The
first few conservation laws in the hierarchy are
\begin{eqnarray*}
\tilde{H}_0=-\frac{9}{2}\int_{\Bbb S} m  d x, \quad
\tilde{H}_1=\frac1{2}\int_{\Bbb S}u^2  d x, \quad
\tilde{H}_2=\int_{\Bbb S}\left
(\frac{3}{2}\mu(u)(A^{-1}\partial_xu)^2+\frac1{6}u^3 \right )  d x.
\end{eqnarray*}
In addition to the peakon solutions same as those of the $\mu$CH
equation, the $\mu$DP equation admits shock-peakon solutions.
$$u=\sum^N_{i=1}[p_i(t)g(x-q_i(t))+s_i(t)g'(x-q_i(t)],$$
where
\begin{equation*}
g'(x)=\left\{
\begin{array}{ll}
0 \quad\quad\quad x=0, \\
x-\frac1{2} \quad x\in(0,1)\end{array}\right.\end{equation*} is the
derivative of $g(x)$ assigning the value zero to the $g'(0)$.

The goal of the present paper is to derive some better conditions of
blow-up solutions and determine blow-up rate for the $\mu$CH and
$\mu$DP equations as well as give new geometric descriptions of
these two equations through invariant curve flows in
centro-equiaffine geometries and pseudo-spherical surfaces or affine
surfaces.

To establish blow-up results in view, we use the method of the
Lyapunov functions $ V(t)=\int_{\mathbb{S}} u_x^3(t, x)  dx $
introduced in \cite{con3} to find some sufficient conditions of
blow-up solutions for the $\mu$CH equation ({\bf Theorem \ref{t4.3}} and
{\bf Theorem \ref{t4.6}}) and the $\mu$DP equation ({\bf Theorem \ref{t4.4}}).
Based on the conservation laws $H_0$, $H_1
$ and $ H_2 $ with the best constant in the Sobolev imbedding $
H^1(\mathbb{S}) \subset L^{\infty}(\mathbb{S}), $ we
are able to  improve significantly blow-up results shown in \cite{khe} and \cite{len2}.

It is noted that  the norm
$\|u(t)\|_{L^{\infty}}$ of the $\mu$DP equation is not uniformly
bounded. To determine a better condition of blow-up solutions, we can  employ the method of
characteristics along a proper choice of a trajectory $ q(t, x) $
defined in (\ref{e1.8}) which captures some zero of the flow $ u(t,
x). $ Using this new method of characteristics together with the
conservation laws, we can derive  an improved blow-up result to guarantee
the slope of the flow tends to negative infinity for the $\mu$DP
equation ({\bf Theorem \ref{t4.7}}) (also see {\bf Theorem
\ref{t4.5}} for the  $\mu$CH equation). This method is also expected to have further
applications in other nonlinear dispersive equations with a part of
the Burgers equation.

The rest of the paper is organized as follows. In Section 2, we
present some properties and estimates of the solutions for the
$\mu$CH and $\mu$DP equations, which will be used for establishing
blow-up results. The main part of the paper, Section 3 is to derive
some precise blow-up scenarios of strong solutions and establish
various results of blow-up solutions with certain initial profiles.
The exact blow-up rate of solutions for these two equations will be
determined in Section 4. Finally in Appendix A, we obtain the
$\mu$CH and $\mu$DP equations again from non-stretching invariant
curve flows in the two-dimensional and three-dimensional
centro-equiaffine geometries, respectively. We also show that both
of equations describe pseudo-spherical surfaces and affine surfaces,
respectively.

 \vskip 0.1cm

\noindent {\it Notation}. Throughout this paper, we identity all
spaces of periodic functions with function spaces over the unit
circle $\mathbb{S}$ in $\mathbb{R}^2$, i. e.
$\mathbb{S}=\mathbb{R}/\mathbb{Z}$. Since all space of functions are
over $\mathbb{S}$, for simplicity, we drop $\mathbb{S}$ in our
notations of function spaces if there is no ambiguity. Throughout the paper, for a given Banach space $Z$, we denote its norm by
$\|\cdot\|_Z$.

\section{Preliminaries}

In this section, we first present the Sobolev-type inequalities
which play a key role to obtain  blow-up results for the
initial-value problem (\ref{e1.1}) and (\ref{e1.2}) in the sequel.
Then based on the first few conservation laws, we will prove some
priori estimates.

\begin{lemma}\label{l2.1}\cite{con2}
If $f \in H^3(\mathbb{S}) $ is such that $\int_{\mathbb{S}}f(x)\ dx=a_0/2,$ then
for every $\varepsilon>0$, we have
$$\max\limits_{x \in \mathbb{S}}f^2(x)\le\dfrac{\varepsilon+2}{24}\int_{\mathbb{S}}f^2_x(x)  d x
+\dfrac{\varepsilon+2}{4\varepsilon}a^2_0.$$
\end{lemma}

\begin{remark}  Since $H^3$ is dense in $H^1,$
Lemma \ref{l2.1} also holds for every $f \in H^1(\mathbb{S})$. Moreover, if
$\int_\mathbb{S}f(x)\ dx=0$, from the
 deduction of this lemma we arrive at the following inequality
\begin{equation}\label{e2.1}
\max\limits_{x\in
\mathbb{S}}f^2(x)\le\dfrac1{12}\int_{\mathbb{S}}f^2_x(x) dx,\quad
x\in\mathbb{S}, \quad f \in H^1(\mathbb{S}).
\end{equation}
\end{remark}

\begin{lemma}\label{l2.2}\cite{but}
For every $\  f(x)\in H^1(a,b)$ periodic and with zero average, i.e.
such that $\int^b_af(x)\ dx=0$, we have
$$\int^b_af^2(x)\ d x\le\left ( \dfrac{b-a}{2\mpi}\right)^2\int^b_a|f'(x)|^2  dx,$$
and equality holds if and only if
$$f(x)=A\cos\left(\dfrac{2\mpi x}{b-a}\right)+B\sin\left(\dfrac{2\mpi x}{b-a}\right).$$
\end{lemma}

\begin{lemma}\label{l2.3}\cite{con5}
Let $T>0$ and $u\in C^1([0,T);H^2(\mathbb{S}))$. Then for every $t\in[0,T)$,
there exists at least one point $\xi(t)\in\mathbb{S}$ with
$w(t):=\inf\limits_{x\in\mathbb{S}}u_x(t,x)=u_x(t,\xi(t))$. The
function $w(t)$ is absolutely continuous on (0,T) with
$$\frac{d w}{d t}=u_{xt}(t,\xi(t)),\quad  a.e. \quad on \quad(0,T).$$
\end{lemma}

\begin{lemma}\label{l2.5}\cite{kat}
If $r>0$, then $H^r\cap L^{\infty}$ is an algebra. Moreover
$$\|fg\|_{H^r}\le c(\|f\|_{L^{\infty}}\|g\|_{H^r}+\|f\|_{H^r}\|g\|_{L^{\infty}}),$$
where $c$ is a constant depending only on $r$.
\end{lemma}

\begin{lemma}\label{l2.6}\cite{kat}
If $r>0$, then
$$\|[\Lambda^r,f]g\|_{L^2}\le c(\|\partial_x f\|_{L^{\infty}}\|\Lambda^{r-1}g\|_{L^2}
+\|\Lambda^r f\|_{L^2}\|g\|_{L^{\infty}}),$$ where $c$ is a constant
depending only on $r$.
\end{lemma}

In the following, we verify some priori estimates for the $\mu$CH equation. Recall that the first two conserved quantities of the $\mu$CH
equation are
$$H_0=\int_{\Bbb S}m\ d x=\int_{\Bbb S}\left (\mu(u)-u_{xx}\right )  d x=\mu(u(t)),$$
and
$$H_1=\frac1{2}\int_{\Bbb S}mu \ d x=\frac1{2}\mu^2(u(t))+\frac1{2}\int_{\Bbb S}u^2_x(t,x)  d x.$$
It is easy to see that  $\mu(u(t))$ and $ \int_{\Bbb S}u^2_x(t,x)   d x $ are conserved in time \cite{khe}. Thus
\begin{equation}\label{2.0}
\mu(u_t)=0.
 \end{equation}

For the sake of convenience, let
\begin{equation}\label{2.1}
\mu_0=\mu(u_0) =\mu(u(t))  =\int_{\mathbb{S}} u(t, x)  d x
 \end{equation}
 and
\begin{equation}\label{2.2}
\mu_1=\left(\int_{\Bbb S}u^2_x(0,x)  d x \right)^{\frac1{2}}=\left(\int_{\Bbb S}u^2_x(t,x)  d x\right)^{\frac1{2}}.
\end{equation}
Then $\mu_0$ and $\mu_1$ are constants and independent of time $t$. Note that
$$\int_{\Bbb S}(u(t,x)-\mu_0)  d x=\mu_0-\mu_0=0.$$
 By Lemma \ref{l2.1}, we find that
$$\max\limits_{x\in \mathbb{S}}\;[u(t,x)-\mu_0]^2\le\dfrac1{12}\int_{\mathbb{S}}u^2_x(t,x)  dx=\dfrac1{12}
\int_{\mathbb{S}}u^2_x(0,x)\ d x=\dfrac1{12}\mu^2_1.$$

From the above estimate, we find that the amplitude of the wave
remains bounded in any time, that is,
$$\|u(t,\cdot)\|_{L^{\infty}}-|\mu_0|\le\|u(t,\cdot)-\mu_0\|_{L^{\infty}}\le\frac{\sqrt{3}}{6}\mu_1,$$
and so
\begin{equation}\label{e2.2}
\|u(t,\cdot)\|_{L^{\infty}}\le|\mu_0|+\frac{\sqrt{3}}{6}\mu_1.
\end{equation}
On the other hand, we have
\begin{equation}\label{e2.3}
\|u(t,x)\|^2_{L^2}=\int_{\Bbb S}u^2(t,x)\ d x\le\|u(t,\cdot)\|^2_{L^{\infty}}\le\left (|\mu_0|+\frac{\sqrt{3}}{6}\mu_1\right )^2.
\end{equation}
It then follows that
\begin{align}\label{e2.4}
\|u(t,\cdot)\|_{H^1}&=\|u(t,\cdot)\|_{L^2}+\|u_x(t,\cdot)\|_{L^2}=\left(\int_{\Bbb S}u^2(t,x) d x\right)^{\frac1{2}}
+\left(\int_{\Bbb S}u^2_x(t,x)\ d x\right)^{\frac1{2}}\nonumber\\
&\le|\mu_0|+\frac{\sqrt{3}}{6}\mu_1+\mu_1=|\mu_0|+\left (1+\frac{\sqrt{3}}{6}\right )\mu_1.
\end{align}

Similar to the conservation law $H_0$ for the $\mu$CH equation, it is easy to see that $ \displaystyle \mu(u_t)=0 $ and
$ \int_{\mathbb{S}} u(t, x) \ d x $ is also conserved in time for the $\mu$DP equation.
Set
$$\mu_0=\mu(u_0) =\mu(u(t)) =\int_{\mathbb{S}} u(t, x)   d x.$$
Since $\tilde{H}_1$ is a
conserved quantity for the $\mu$DP equation, we define that
\begin{equation}\label{2.3}
\mu_2=\left(\int_{\Bbb S}u^2(0,x)  d x\right)^{\frac1{2}}=\left(\int_{\Bbb S}u^2(t,x)  dx \right)^{\frac1{2}}.
\end{equation}
Then $\mu_2$ is a constant and independent of time $t$.

Recall in \cite{khe} that the mean of any solution $u(t,x)$ is
conserved by the flow and hence the initial value problem
(\ref{e1.1}) and (\ref{e1.2}) can be recast in the following.
\begin{equation}\label{e1.4}
\left\{
 \begin{array}{ll}
\begin{split}
&u_t+uu_x+A^{-1}\partial_x \left (\lambda\mu_0u+\frac{3-\lambda}{2}u^2_x \right )=0, \qquad t > 0, \quad x \in \mathbb{R}, \\
&u(0,x)=u_0(x),  \qquad x \in \mathbb{R},\\
&u(t,x+1)=u(t,x),\quad\quad\quad t \ge 0, \quad x\in\mathbb{R},
 \end{split}
\end{array} \right.
\end{equation}
with $ \lambda = 2 $ and $ \lambda = 3, $ respectively, where $A=\mu-\partial^2_x$ is an isomorphism between $H^s(\mathbb{S})$ and $H^{s-2}(\mathbb{S})$ with the inverse
$v=A^{-1}w$ given explicitly by
 \begin{align}
v(x)=&\left (\frac{x^2}{2}-\frac{x}{2}+\frac{13}{12} \right )\mu(w)+\left (x-\frac1{2}\right )\int^1_0\int^y_0w(s)  ds dy ,\nonumber\\
&-\int^x_0\int^y_0w(s)\ d sd y+\int^1_0\int^y_0\int^s_0w(r)\ dr ds d y .\label{e1.5}
\end{align}
Since $A^{-1}$ and $\partial_x$ commute, the following identities hold
\begin{equation}\label{e1.6}
A^{-1}\partial_xw(x)=\left (x-\frac1{2}\right )\int^1_0w(x)  d x-\int^x_0w(y)  dy+\int^1_0\int^x_0w(y)  d y d x,
\end{equation}
and
\begin{equation}\label{e1.7}
A^{-1}\partial^2_xw(x)=-w(x)+\int^1_0 w(x)  dx.
\end{equation}

The local well-posedness results of the initial value problem
(\ref{e1.4})  is already established in \cite{khe}
and \cite{len2}.
\begin{prop}\label{p2.1}
Let $u_0\in H^s(\mathbb{S})$, $s>3/2$. Then there exist a maximal life span $T>0$ and a unique
solution $ u $ to (\ref{e1.4}) such that
$$
u \in C([0,T);H^s(\mathbb{S}))\cap C^1([0,T);H^{s-1}(\mathbb{S}))
$$
which depends continuously on the initial data $u_0$.
\end{prop}
\begin{remark}
The maximal time of existence $ T > 0 $ in Proposition \ref{p2.1} is independent of the Sobolev index $ s > 3/2.$ See Yin \cite{yin2} for an adaptation of Kato method \cite{kat1} to the proof of this statement.
\end{remark}

In \cite{khe}  and \cite{len2}, the authors also showed  that
the $\mu$CH and $\mu$DP equations admit global (in time) solutions
and a blow-up mechanism. It is our purpose here to derive  the precise
scenarios and initial  conditions guaranteeing the blow-up of
strong solutions to the initial-value problem  (\ref{e1.1}) and
(\ref{e1.2}), which will significantly improve the results in \cite{khe} and
\cite{len2}. In the case of $\mu_0=0$, these two equations reduce to
the Hunter-Saxon equation and $\mu$Burgers equation respectively. Since these two special cases have recently
been the object of intensive study (\cite {bc, hun1, hun2, len2, yin3}, for example),  we only focus on the case of $\mu_0\neq0$ in the rest of the paper.

Given a solution $u(t,x)$  of the initial value problem (\ref{e1.4}) with initial data $u_0$, we let $t\rightarrow q(t,x)$
be the flow of $u(t,x)$, that is
\begin{equation}\label{e1.8}
\left\{
 \begin{array} {cc}
 \begin{split}
&\dfrac{d q(t,x)}{d t}=u(t,q(t,x)), \qquad  t > 0, \quad x \in \mathbb{R},\\
&q(0,x)=x.\\
\end{split}
\end{array} \right.
\end{equation}
A direct calculation shows that  $q_{tx}(t,x)=u_x(t,q(t,x))q_x(t,x)$. Hence, for $t>0,x\in\mathbb{R}$, we have
$$q_x(t,x)=\me^{\int^t_0u_x(\tau,q(\tau,x))\ d\tau}>0,$$
which implies that $q(t,\cdot):\mathbb{R}\rightarrow\mathbb{R}$ is a diffeomorphism of the line for
every $t\in[0,T)$. This is inferred
that the $L^{\infty}$-norm of any function $v(t,\cdot)\in L^{\infty}, \;  t\in[0,T)$ is preserved under
the family of diffeomorphisms
$q(t,\cdot)$ with $t\in[0,T)$, that is,
$$\|v(t,\cdot)\|_{L^{\infty}}=\|v(t,q(t,\cdot))\|_{L^{\infty}},\quad t\in[0,T).
$$
Consider the  $\mu$-version (\ref{e1.4})
with  $m=Au$, $ A = \mu - \partial_x^2. $ It is easy to verify that
 at each point of the circle the solution $u(t,x)$
satisfies a local conservation law
$$m(t,q(t,x))(\partial_xq(t,x))^\lambda=m_0(x)=\mu(u_0(x))-u''_0(x),
$$where $ u''_0(x) $ is the second derivative of $ u_0(x) $ with respective to $ x. $

Unlike the $ \mu$CH equation,  $ \displaystyle \|u(t)\|_{L^{\infty}} $  of the solution $ u(t,x) $ of  the
$\mu$DP equation is not uniformly bounded for $ t.$  However,  we are able to establish an important estimate in the following.
\begin{lemma}\label{l2.7}
Assume $u_0\in H^s,s>3/2$. Let $T$ be the maximal existence time of the solution $u(t,x)$ to the
initial value problem (\ref{e1.2}) associated with the $\mu$DP equation.  Then we have
$$\|u(t,x)\|_{L^{\infty}}\le\left (\frac{3}{2}\mu_0^2+6|\mu_0|\mu_2\right )t+\|u_0\|_{L^{\infty}},\quad \forall t\in[0,T].$$
\end{lemma}
\begin{proof}
Since the existence time $T$ is independent of the choice of $s$ by
Proposition \ref{p2.1}, applying a simple density argument, we only
need to consider the case $s=3$. Let $T$ be the maximal existence
time of the solution $u(t,x)$ to the initial value problem
(\ref{e1.2}) with the initial data $u_0\in H^3(\mathbb{S})$. By
(\ref{e1.4}) with $ \lambda = 3,$  the first equation of the initial
value problem (\ref{e1.2}) is equivalent to the following equation.
$$u_t+uu_x=-3\mu_0A^{-1}\partial_xu.$$
In view of (\ref{e1.6}), we have
$$|A^{-1}\partial_xu|\le\frac1{2}|\mu_0|+2\mu_2.$$
On the other hand, it follows from (\ref{e1.8}) that
$$\dfrac{d u(t,q(t,x))}{d t}=u_t(t,q(t,x))+u_x(t,q(t,x))\dfrac{d q(t,x)}{d t}=(u_t+uu_x)(t,q(t,x)).$$
Combining the above two estimates yields
$$-\left (\frac{3}{2}\mu_0^2+6|\mu_0|\mu_2\right )\le\dfrac{d u(t,q(t,x))}{d t}\le\frac{3}{2}\mu_0^2+6|\mu_0|\mu_2.$$
Integrating the above inequality with respect to $t<T$ on $[0,t],$ one easily finds
$$-\left (\frac{3}{2}\mu_0^2+6|\mu_0|\mu_2\right )t+u_0(x)\le u(t,q(t,x))\le\left (\frac{3}{2}\mu_0^2+6|\mu_0|\mu_2\right )t+u_0(x).$$
This thus implies  that
$$|u(t,q(t,x))|\le\|u(t,q(t, \cdot))\|_{L^{\infty}}\le\left (\frac{3}{2}\mu_0^2+6|\mu_0|\mu_2\right )t+\|u_0(x)\|_{L^{\infty}}.$$
In view of the diffeomorphism property of $q(t,\cdot),$  we obtain
$$\|u(t, \cdot)\|_{L^{\infty}}=\|u(t,q(t, \cdot))\|_{L^{\infty}}\le\left (\frac{3}{2}\mu_0^2+6|\mu_0|\mu_2\right )t
+\|u_0(x)\|_{L^{\infty}}.$$
This completes the proof of Lemma \ref{l2.7}.
\end{proof}

\section{Blow-up solutions}

In this section, we establish the precise blow-up scenarios and give
sufficient conditions for blow-up of solutions to the initial value
problem (\ref{e1.1}) and (\ref{e1.2}). Indeed, we determine the
precise blow-up scenarios for the problem (\ref{e1.4}) in the
following.

\begin{theorem}\label{t4.1}
Suppose that  $ \lambda \in \mathbb{R}. $ Let $u_0\in H^s(\mathbb{S}), s>3/2$ be given and assume that $T$ is the maximal existence
time of the corresponding solution $u(t,x)$ to the initial value problem (\ref{e1.4}) with the initial data
$u_0$. If there exists $M>0$ such that
$$\|u_x (t) \|_{L^{\infty}}\le M,\quad t\in[0,T),$$
 then the $H^s$-norm of $u(t,\cdot)$ does not blow up on $[0,T)$.
\end{theorem}
\begin{proof}
We always assume that $c$ is a generic positive constant. Let $\Lambda=(1-\partial^2_x)^{1/2}$.
Applying the operator $\Lambda^s$ to the first equation in (\ref{e1.4}), then
multiplying by $\Lambda^s u$ and integrating over $\mathbb{S}$ with respect to $x$ lead to
\begin{equation*}
\frac{d}{d t}\|u\|^2_{H^s}=-2(uu_x,u)_{H^s}-2\left (u,A^{-1}\partial_x \left (\lambda\mu_0u+\frac{3-\lambda}{2}u^2_x \right )\right )_{H^s}.
\end{equation*}
Let us estimate the right hand side of the above equation.
\begin{align*}
|(uu_x,u)_{H^s}|&=|(\Lambda^s(uu_x),\Lambda^s
u)_{L^2}| =|([\Lambda^s,u]u_x,\Lambda^s
u)_{L^2}+(u\Lambda^s
u_x,\Lambda^s u)_{L^2}|\\
&\le\|[\Lambda^s,u]u_x\|_{L^2}\|\Lambda^s
u\|_{L^2}+\frac1{2}|(u_x\Lambda^s u,\Lambda^s u)_{L^2}|\\
 &\le c\|u_x\|_{L^{\infty}}\|u\|^2_{H^s}.
\end{align*}
In the above inequality, we used Lemma \ref{l2.6} with $r=s$.

By the expression (\ref{e1.6}) and (\ref{e1.7}), one finds that
\begin{align*}
\|A^{-1}\partial_xu\|_{H^s}&\le\|A^{-1}\partial_xu\|_{L^2}+\|\partial_xA^{-1}\partial_xu\|_{H^{s-1}}\\
&\le3\|u\|_{L^2}+\left\|-u+\int_{\mathbb{S}}u\ d x\right\|_{H^{s-1}}\\
&\le3\|u\|_{L^2}+2\|u\|_{H^{s-1}}\le5\|u\|_{H^{s-1}},
\end{align*}
and
\begin{align*}
\|A^{-1}\partial_xu_x^2\|_{H^s}&\le\|A^{-1}\partial_xu_x^2\|_{L^2}+\|\partial_xA^{-1}\partial_xu_x^2\|_{H^{s-1}}\\
&\le3\|u_x\|^2_{L^2}+\left\|-u_x^2+\int_{\mathbb{S}}u_x^2\ d x\right\|_{H^{s-1}}\\
&\le3\|u_x\|^2_{L^2}+2\|u_x^2\|_{H^{s-1}}\le5\|u_x\|_{L^{\infty}}\|u\|_{H^s},
\end{align*}
where in the last step we used Lemma \ref{l2.5} with $r=s-1$. It
then follows that
\begin{align*}
|(u,A^{-1}\partial_xu)_{H^s}|
\le
c\|u\|_{H^s}\|A^{-1}\partial_xu\|_{H^s}\le c\|u\|^2_{H^s}
\end{align*}and
\begin{align*}
|(u,A^{-1}\partial_xu^2_x)_{H^s}|
\le
c\|u\|_{H^s}\|A^{-1}\partial_xu^2_x\|_{H^s}\le c\|u_x\|_{L^{\infty}}\|u\|^2_{H^s}.
\end{align*}

Combining the above three estimates, we obtain
\begin{align*}
&\dfrac{d}{d t}\|u\|^2_{H^s}\le c(1+\|u_x\|_{L^{\infty}})\|u\|^2_{H^s}.
\end{align*}
An application of Gronwall's inequality and the assumption of the
theorem lead to
$$\|u\|^2_{H^s}\le\exp(c(1+M)t)\|u_0\|^2_{H^s}.$$
This completes the proof of Theorem \ref{t4.1}.
\end{proof}

\begin{theorem}\label{t4.2}
Let $u_0\in H^s({\mathbb S}), s > 3/2$, and $u(t,x)$ be the
solution of the initial value problem (\ref{e1.4}) with life-span
$T$. Then $T$ is finite if and only if
$$\underset{t\uparrow T}{\liminf}\left \{\underset{x\in \mathbb{S}}{\inf}
[(2\lambda -1)u_x(t,x)]\right \}=-\infty.$$
\end{theorem}
\begin{proof}
Since the existence time $T$ is independent of the choice of $s$
by Proposition \ref{p2.1}, applying a simple
density argument, we only need to consider the case $s=3$.
Multiplying the $\mu$-version of the family (\ref{f1.1}) by $m$ and
integrating over $\mathbb{S}$ with respect to $x$ yield
\begin{align*}
\dfrac{d}{d t}\int_{\mathbb{S}}m^2\;d
x=&-2\lambda\int_{\mathbb{S}}u_xm^2\;d x-2\int_{\mathbb{S}}umm_x\;d
x\\
=&(1-2\lambda)\int_{\mathbb{S}}u_xm^2\;d x.
\end{align*}
If $(2\lambda -1)u_x$ is bounded from below on $[0,T)\times {\mathbb
S}$, i.e., there exists $N>0$ such that $(2\lambda -1)u_x\geq -N$ on
$[0,T)\times {\mathbb S}$, then it is thereby inferred from the
above estimate that
\begin{eqnarray*}
\frac{d}{dt}\int_{\mathbb S}m^2\;dx\leq N\int_{\mathbb S}m^2\ dx.
\end{eqnarray*}
Applying Gronwall's inequality then yields for $t\in [0,T)$
\begin{eqnarray*}
\int_{\mathbb S}m^2\ dx\leq e^{Nt}\int_{\mathbb S}m_0^2\ dx.
\end{eqnarray*}
Note that $\mu(u)$ is independent of $t$ to the $\mu$-version of the
family (\ref{f1.1}) for any $\lambda\in {\mathbb R}$. Thus we have
\begin{eqnarray*}
\int_{\mathbb S}m^2(t,x)\ dx=\mu^2(u)+\int_{\mathbb
S}u_{xx}^2\;dx=\mu_0^2+||u_{xx}||_{L^2}^2.
\end{eqnarray*}
It then follows from Sobolev's embedding $H^1\hookrightarrow
L^{\infty}$ and Remark 2.1 that for $t\in [0,T)$
\begin{eqnarray*}
||u_x||_{L^{\infty}}\leq \frac {1}{2\sqrt{3}}||u_{xx}||_{L^2}\leq
\frac {1}{2\sqrt{3}}||m||_{L^2}\leq \frac {1}{2\sqrt{3}}e^{\frac 12
NT}||m_0||_{L^2}. \end{eqnarray*} As a result of Theorem \ref{t4.1},
we deduce that the solution exists globally in time.

On the other hand, if the slope of the solution becomes unbounded
from below, by the existence  of the local strong solution
Proposition \ref{p2.1} and Sobolev's embedding theorem, we infer
that the solution will blow-up in finite time. The proof of Theorem
\ref{t4.2} is then completed.
\end{proof}

\begin{remark} In particular, when $\lambda=1/2$, the solution $ u $  to  (\ref{e1.4})
exists in $ H^s(\mathbb{S}), s > 3/2 $ globally in time.
\end{remark}

In the following, by means of the blow-up scenarios we establish some sufficient conditions
guaranteeing the development of singularities.

\subsection {Blow-up for the $\mu$CH equation}

We are now in a position to give the first blow-up result for the $\mu$CH equation.
\begin{theorem}\label{t4.3}
Let $u_0\in H^s(\mathbb{S}), s > 3/2$ and $T>0$ be the maximal time of
existence of the corresponding solution $u(t,x)$ to (\ref{e1.1})
with the initial data $u_0$. If $\
(\sqrt{3}/\mpi)|\mu_0|<\mu_1,$ where $\mu_0 $ and $ \mu_1$ are
defined in (\ref{2.1}) and (\ref{2.2}), then the corresponding
solution $u(t,x)$ to (\ref{e1.1}) associated with the $\mu$CH
equation must blow up in finite time  $T$ with
$$0<T\le\inf\limits_{\alpha \in I}\left (\dfrac{6}{1-6\alpha|\mu_0|}+4\mpi^2\alpha\dfrac{1+|\int_{\mathbb{S}}u_{0x}^3(x)\;d x|}
{6\mpi^2\alpha\mu_1^4-3|\mu_0|\mu_1^2}\right ) $$
where $ I = \left (\frac{|\mu_0|}{2\mpi^2\mu_1^2},\frac1{6|\mu_0|} \right )$,
such that
$$\liminf\limits_{t\uparrow T}\left (\inf\limits_{x\in\mathbb{S}}u_x(t,x)\right )=-\infty.$$
\end{theorem}
\begin{proof}
As discussed above, it suffices to consider the case $s=3$. Since the case $\mu_0=0$ was proved in \cite{khe},
we only need to discuss  the case $\mu_0\neq0$. In this case, $ \mu_1 > 0.$
Differentiating the $\mu$CH equation with respect to $x$ yields
$$u_{tx}+u^2_x+uu_{xx}+A^{-1}\partial^2_x \left (2u\mu_0+\frac1{2}u^2_x \right )= 0,$$
In view of (\ref{2.1}), (\ref{2.2}) and (\ref{e1.7}), we have
\begin{align}\label{e4.0}
u_{tx}=-\frac1{2}u^2_x-uu_{xx}+2u\mu_0-2\mu_0^2-\frac1{2}\mu_1^2.
\end{align}
Multiplying (\ref {e4.0}) by $3u^2_x$ and integrating on $\mathbb{S}$ with respect to $x$, we obtain for any $ t \in [0, T) $ that
 \begin{align}\label{e3.1}
 &\dfrac{d}{d t}\int_{\mathbb{S}}u_x^3\;d x=\int_{\mathbb{S}}3u_x^2u_{xt}\;d x\nonumber\\
 &\quad=-\frac{3}{2}\int_{\mathbb{S}}u_x^4\;d x-\int_{\mathbb{S}}3uu_x^2u_{xx}\;d x+6\mu_0\int_{\mathbb{S}}uu_x^2\;d x-
 6\mu^2_0\int_{\mathbb{S}}u_x^2\;d x-\frac{3}{2}\left (\int_{\mathbb{S}}u_x^2\;d x\right )^2\nonumber\\
 &\quad=-\frac1{2}\int_{\mathbb{S}}u_x^4\;d x-\frac{3}{2}\mu_1^4+6\mu_0\int_{\mathbb{S}}(u-\mu_0)u_x^2\;d x.
 \end{align}
 On the other hand, it follows from Lemma \ref{l2.2} for any $ \alpha > 0 $ that
 \begin{align*}
 \int_{\mathbb{S}}(u-\mu_0)u_x^2\;d x&\le\left (\int_{\mathbb{S}}(u-\mu_0)^2\;d x\right )^{\frac1{2}}
 \left (\int_{\mathbb{S}}u_x^4\;d x\right )^{\frac1{2}}\\
 &\le\frac{\alpha}{2}\int_{\mathbb{S}}u_x^4\;d x+\frac1{2\alpha}\int_{\mathbb{S}}(u-\mu_0)^2\;d x\\
 &\le\frac{\alpha}{2}\int_{\mathbb{S}}u_x^4\;d x+\frac1{8\mpi^2\alpha}\int_{\mathbb{S}}u_x^2\;d x.
 \end{align*}
 Therefore we deduce that
 $$\dfrac{d}{d t}\int_{\mathbb{S}}u_x^3\;d x\le \left (3\alpha|\mu_0|-\frac1{2}\right )\int_{\mathbb{S}}u_x^4
 \;d x-\frac{3}{2}\mu_1^4+\frac{3}{4\mpi^2\alpha}|\mu_0|\mu_1^2.$$
 By the assumption of the theorem, we know that
 $ {|\mu_0|}/({2\mpi^2\mu_1^2})<1/({6|\mu_0|}). $ Let   $\alpha>0$ satisfy
 $$\frac{|\mu_0|}{2\mpi^2\mu_1^2}<\alpha<\frac1{6|\mu_0|}.$$
This in turn implies that $ \displaystyle 6\alpha|\mu_0|-1<0 $ and $
\displaystyle 2\pi^2\alpha\mu_1^2-|\mu_0|>0.$ Define $ c_1 $ and $
c_2 $ by
 $$c_1=\frac1{2}-3\alpha|\mu_0|>0,\quad c_2=\frac{3}{2}\mu_1^4-\frac{3}{4\mpi^2\alpha}|\mu_0|\mu_1^2>0.$$
It is then clear that
$$
\dfrac{d}{d t}\int_{\mathbb{S}}u_x^3\;d x\le -c_1\int_{\mathbb{S}}u_x^4\;d x-c_2\le
 -c_1\left (\int_{\mathbb{S}}u_x^3\;d x\right )^{\frac{4}{3}}-c_2.
 $$
 Let  $V(t)=\int_{\mathbb{S}}u_x^3(t, x)\ d x$ with $ t \in [0, T).$ Then the above inequality can be rewritten as
 $$\dfrac{d}{d t}V(t)\le-c_1(V(t))^{\frac{4}{3}}-c_2\le-c_2<0,\quad t\in[0,T).$$
This implies that $V(t)$ decreases strictly in $[0, T). $  Let $t_1=(1+|V(0)|)/c_2.$ One can assume $ t_1 < T. $ Otherwise, $ T \le t_1 < \infty $ and the theorem is proved. Now integrating the above inequality over $[0,t_1]$ yields
 $$V(t_1)=V(0)+\int^{t_1}_0\dfrac{d}{d t}V(t) \ d t \le|V(0)|-c_2t_1\le-1.$$
 It is also found that
 $$\dfrac{d}{d t}V(t)\le-c_1(V(t))^{\frac{4}{3}},\qquad t\in[t_1,T),$$
 which leads to
 $$-3\dfrac{d}{d t}\left (\dfrac1{(V(t))^{\frac1{3}}}\right )=(V(t))^{-\frac{4}{3}}\dfrac{d}{d t}V(t)\le-c_1,\qquad t\in[t_1,T).$$
 Integrating both sides of the above inequality and applying $V(t_1)\le-1$ yield
 $$-\dfrac1{(V(t))^{\frac1{3}}}-1\le-\dfrac1{(V(t))^{\frac1{3}}}+\dfrac1{(V(t_1))^{\frac1{3}}}\le-\frac{c_1}{3}(t-t_1),\qquad t\in[t_1,T).$$
Recall that $V(t)\le V(t_1)\le-1$ in $[t_1,T)$. It follows that
 $$V(t)\le\left [\dfrac{3}{c_1(t-t_1)-3}\right ]^3\rightarrow-\infty, \quad\text{as}\quad t\rightarrow t_1+\frac{3}{c_1}.
 $$
On the other hand, we have
 $$\int_{\mathbb{S}}u_x^3\;d x\ge\inf\limits_{x\in\mathbb{S}}u_x(t,x)\int_{\mathbb{S}}u_x^2\;d x=\mu_1^2\inf\limits_{x\in\mathbb{S}}u_x(t,x).$$
 This then implies that $0<T \le t_1+3/c_1$ such that
 $$\liminf\limits_{t\uparrow T}\left (\inf\limits_{x\in\mathbb{S}}u_x(t,x)\right )=-\infty.$$
 This  completes the proof of Theorem \ref{t4.3}.
\end{proof}

\begin {remark}  Note that in \cite{khe}, the initial condition of blow-up mechanism is $4|\mu_0|\le\mu_1$. Therefore, Theorem \ref{t4.3}  improves the blow-up result in \cite{khe}.
\end{remark}
In the case $ (\sqrt{3}/\mpi)|\mu_0|\ge\mu_1$, we have the
following blow-up result.
\begin{theorem} \label{t4.5}
Let $u_0 \in H^s(\mathbb{S}),  s > 3/2$ and $T>0$ be the maximal time of
existence of the corresponding solution $u(t,x)$ to (\ref{e1.1})
with the initial data $u_0$. If  $ (\sqrt{3}/\mpi)|\mu_0|\ge\mu_1$
and
$$ \inf_{x \in \mathbb{S}}  u'_0(x) <-\sqrt{2\mu_1 \left (\frac{\sqrt{3}}{3}|\mu_0|-\frac1{2}\mu_1 \right )}:\equiv-K,$$ where $ u'_0(x) $ is the derivative of $ u_0(x) $ with respective to $ x, $
then the corresponding
solution $u(t,x)$ to (\ref{e1.1}) blows up in finite time  $T$ with
$$0<T\le \left(-\dfrac{2}{\inf_{x \in\mathbb{S}} u'_0(x )+\sqrt{-K \inf_{x \in \mathbb{S}} u'_0(x)}}\right),$$
such that
$$\liminf\limits_{t\uparrow T}\left (\inf\limits_{x\in\mathbb{S}}u_x(t,x)\right )=-\infty.$$
\end{theorem}
\begin{proof}As discussed above, it suffices to consider the case $s=3$. Note that the assumption $\  (\sqrt{3}/\mpi)|\mu_0|\ge\mu_1$ implies that $\  (2/\sqrt{3})|\mu_0| > \mu_1$. Therefore the non-negative constant $K$ is well-defined.
In view of  (\ref{e4.0}), we have
\begin{align*}
u_{tx}+uu_{xx}=-\frac1{2}u^2_x+2u\mu_0-2\mu_0^2-\frac1{2}\mu^2_1.
\end{align*}
By Lemma \ref{l2.3}, there is $ x_0 \in \mathbb{S} $ such that $ \displaystyle u_0'(x_0) = \inf_{x \in \mathbb{S}} u_0'(x). $
Define $w(t)=u_x(t,q(t,x_0))$, where $q(t,x_0)$ is the flow of $u(t,q(t,x_0))$. Then
$$\dfrac{d}{d t}w(t)=(u_{tx}+u_{xx}q_t)(t,q(t,x_0))=(u_{tx}+uu_{xx})(t,q(t,x_0)).$$
Substituting $(t,q(t,x_0))$ into the above equation and using (\ref{e2.1}), we obtain
\begin{align*}
\dfrac{d}{d t}w(t)&=-\frac1{2}w^2(t)+2u\mu_0(t,q(t,x_0))-2\mu_0^2-\frac1{2}\mu^2_1\\
&\le-\frac1{2}w^2(t)+2\mu_0|u(t,q(t,x_0))-\mu_0|-\frac1{2}\mu^2_1\\
&\le-\frac1{2}w^2(t)+\mu_1\left (\frac{\sqrt{3}}{3}|\mu_0|-\frac1{2}\mu_1\right )\\
&=-\frac1{2}w^2(t)+\frac1{2}K^2.
\end{align*}
By the assumption $w(0)<-K$, it follows that  $w'(0)<0$ and $w(t)$
is  strictly decreasing  on $[0,T)$. Set
$$\delta=\frac1{2}-\frac1{2}\sqrt{\frac{K}{-u'_0(x_0)}}\in \left (0,\frac1{2} \right ).$$
And so
$$(u'_0(x_0))^2=\frac{K^2}{(1-2\delta)^4}<w^2(t),$$
which is to say
$$K^2<(1-2\delta)^4w^2(t).$$
Therefore
$$\dfrac{d}{d t}w(t)\le-\frac1{2}w^2(t)[1-(1-2\delta)^4]=-\delta w^2(t), \quad t\in[0,T),$$
which leads to
$$-\dfrac{d}{d t}\frac1{w(t)}=\frac1{w^2(t)}\dfrac{d}{d t}w(t)\le-\delta,\quad t\in[0,T).$$
Integrating both sides over $[0,t)$ yields
$$-\frac1{w(t)}+\frac1{u'_0(x_0)}\le-\delta t,\quad t\in[0,T).$$
So
$$w(t)\le\dfrac{u'_0(x_0)}{1+\delta t u'_0(x_0)}\rightarrow-\infty, \quad\text{as}\quad t\rightarrow -\frac1{\delta u'_0(x_0)}.$$
This implies
$$T\le-\frac1{\delta u'_0(x_0)}<+\infty.$$ In consequence, we have
$$\liminf\limits_{t\uparrow T}\left (\inf\limits_{x\in\mathbb{S}}u_x(t,x)\right )=-\infty.$$
This  completes the proof of Theorem \ref{t4.5}.
\end{proof}

\begin{remark} We can apply Lemma \ref{l2.3} to verify the above theorem under the same conditions.
In fact, if we define $w(t)=u_x(t,\xi(t))=\inf\limits_{x\in\mathbb{S}}[u_x(t,x)]$,
then for all $t\in[0,T)$, $u_{xx}(t,\xi(t))=0$. Thus if $(\sqrt{3}/\mpi)|\mu_0|\ge\mu_1$, one finds that
$$
\frac{d}{dt} w(t)\le-\frac1{2}w^2(t)+\frac1{2}K^2,$$
where $K$ is the same as  Theorem \ref{t4.5}. Then by means of the assumptions of Theorem \ref{t4.5} and following the
lines of the proof of Theorem \ref{t4.5}, we see that if
$$w(0)<-\sqrt{2\mu_1\left(\frac{\sqrt{3}}{3}|\mu_0|-\frac1{2}\mu_1\right)},$$ then $T$ is
finite and $\liminf\limits_{t\uparrow T}\left(\inf\limits_{x\in\mathbb{S}}u_x(t,x)\right)=-\infty$.
\end{remark}

Using the conserved quantities $H_2$, we can derive the following blow-up result.
\begin{theorem}\label{t4.6}
Let $u_0\in H^s(\mathbb{S}),  s > 3/2$ and $T>0$ be the maximal time of existence of the corresponding
solution $u(t,x)$ to (\ref{e1.1}) with the initial data $u_0$. If $\mu_1^4+4\mu^2_0\mu^2_1>8\mu_0H_2$ (in particular, $ \mu_0 H_2 \le 0$),
where $\mu_0,\mu_1$ are defined in (\ref{2.1}) and (\ref{2.2}). Then the corresponding
solution $u(t,x)$ to (\ref{e1.1}) blows up in finite time $T$ with
$$0<T\le6+\dfrac{1+\left|\int_{\mathbb{S}}u_{0x}^3(x)\;d x\right|}
{\frac{3}{2}\mu_1^4+6\mu^2_0\mu_1^2-12\mu_0H_2},$$
such that
$$\liminf\limits_{t\uparrow T}\left(\inf\limits_{x\in\mathbb{S}}u_x(t,x)\right)=-\infty.$$
\end{theorem}
\begin{proof}
Again it suffices to consider the case $s=3$. Recall that
$$H_2=\int_{\Bbb S}\left (\mu_0u^2+\frac1{2}uu_x^2 \right )d x $$ is independent of time $ t. $
In view of  (\ref{e3.1}), we obtain
 \begin{align*}
 \dfrac{d}{d t}\int_{\mathbb{S}}u_x^3\;d x&=-\frac1{2}\int_{\mathbb{S}}u_x^4\;d x-\frac{3}{2}\mu_1^4
 +6\mu_0\int_{\mathbb{S}}uu_x^2\;d x-6\mu^2_0\int_{\mathbb{S}}u_x^2\;d x\\
 &=-\frac1{2}\int_{\mathbb{S}}u_x^4\;d x-\frac{3}{2}\mu_1^4+12\mu_0H_2-12\mu_0^2\int_{\mathbb{S}}u^2\;d x-6\mu_0^2\mu_1^2\\
 &\le-\frac1{2}\int_{\mathbb{S}}u_x^4\;d x-\frac{3}{2}\mu_1^4+12\mu_0H_2-6\mu_0^2\mu_1^2.
 \end{align*}
  By the assumption of the theorem, we have that $ \displaystyle \mu_1^4+4\mu^2_0\mu^2_1>8\mu_0H_2. $
 Let
 $$c_1=\frac1{2}>0,\quad c_2=\frac{3}{2}\mu_1^4+6\mu^2_0\mu^2_1-12\mu_0H_2>0.$$
 It then follows that
 $$\dfrac{d}{d t}\int_{\mathbb{S}}u_x^3\;d x\le -c_1\int_{\mathbb{S}}u_x^4\;d x-c_2\le
 -c_1\left (\int_{\mathbb{S}}u_x^3\;d x\right )^{\frac{4}{3}}-c_2.$$
Define  $V(t)=\int_{\mathbb{S}}u_x^3(t, x)\ d x$ with $ t \in [0, T).$ It is clear that
 $$\dfrac{d}{d t}V(t)\le-c_1(V(t))^{\frac{4}{3}}-c_2\le-c_2<0,\quad t\in[0,T).$$
Let $t_1=(1+|V(0)|)/c_2$. Then following the proof  of  Theorem \ref{t4.3},  we have
 $$T\le t_1+\frac{3}{c_1}<+\infty.$$
 This implies the desired result as in Theorem \ref{t4.6}.
\end{proof}

\subsection {Blow-up for the $\mu$DP equation}

The first blow-up result for the $\mu$DP equation is given in the following.

\begin{theorem}\label{t4.7}
Let $u_0 \in H^s(\mathbb{S}), s > 3/2$ and $T>0$ be the maximal time of existence of the corresponding
solution $u(t,x)$ to (\ref{e1.2}) with the initial data $u_0$. If $\mu_0\tilde{H}_2 \le 0$, i.e.
$ \displaystyle \int_{\mathbb{S}} \left ( \frac{3}{2} \mu_0^2 (A^{-1} \partial_xu_0)^2 + \frac{\mu_0}{6} u_0^3 \right )\ dx  \le 0, $ then the corresponding
solution $ u(t, x) $ to (\ref{e1.2}) associated with the $\mu DP $ equation must blow up in finite time $ T > 0 $.
 If $ \mu_0 \neq 0, $ then
$$0<T\le1+\dfrac{1+|u'_0(\xi_0))|}{3\mu^2_0},$$
such that
$$\liminf\limits_{t\uparrow T}(\inf\limits_{x\in\mathbb{S}}u_x(t,x))=-\infty $$
where $ \mu_0 u_0(\xi_0) \le 0. $
\end{theorem}
\begin{proof} Since the case $ \mu_0 = 0 $ was proved in \cite {len2}, we only need to show the case $ \mu_0 \neq 0. $  Again as discussed previously, it suffices to prove the theorem only with  $ s = 3.$   To this end, by the assumption $ \mu_0 \tilde{H}_2 \le 0, $ we firstly want to show there exists some $ \xi(t) \in \mathbb{S} $ for any fixed $ t \in [0, T) $ with the maximum existence time $ T > 0 $ and $ \xi(0) = \xi_0 $ such that $ \mu_0 u(t, \xi(t)) \le 0 $ for any fixed $ t \in [0, T). $  If not, then $ \mu_0 u(t, x) > 0 $ with any $ x \in \mathbb{S} $ and some $ t \in [0, T). $  Recall that
$$
\tilde{H}_2= \int_{\Bbb S}\left (\frac{3}{2}\mu(u) (A^{-1}\partial_xu)^2+\frac1{6}u^3 \right )\ dx = \int_{\Bbb S}\left (\frac{3}{2}\mu_0(A^{-1}\partial_xu_0)^2+\frac1{6}u_0^3 \right )\ dx
$$
and $  \mu (u(t)) = \mu(u_0) = \mu_0 $ are independent of time $ t. $  It then follows that
$$
\mu_0 \tilde{H}_2= \int_{\Bbb S}\left (\frac{3}{2}\mu^2_0 (A^{-1}\partial_xu )^2+\frac1{6} (\mu_0 u) u^2 \right )\ dx > 0
$$
which contradicts the assumption $ \mu_0 \tilde H_2 \le 0. $

Next we take the trajectory $ q(t, x) $ defined in (\ref{e1.8}). Since $ q(t,\cdot):$   $ \mathbb{R} \to \mathbb{R} $ is a diffeomorphism for every $ t \in [0, T). $ It is then inferred that there is $ x_0(t) \in \mathbb{R} $ such that
$$
q(t, x_0(t)) = \xi(t), \qquad t \in [0, T)
$$ with $ x_0(0) = \xi_0. $
Define $w(t)=u_x(t,q(t,x_0(t)))$, where $q(t,x_0(t))$ is the flow of $u(t,q(t,x_0(t)))$.
Differentiating the $\mu$DP equation with respect to $x$ yields
$$u_{tx}+u^2_x+uu_{xx}+3\mu_0A^{-1}\partial^2_xu=0,$$
from (\ref{e1.7}) we can deduce that
\begin{align}\label{e4.1}
u_{tx}=-u^2_x-uu_{xx}+3\mu_0(u-\mu_0).
\end{align}
 However
$$\dfrac{d}{d t}w(t)=(u_{tx}+u_{xx}q_t)(t,q(t,x_0(t)))=(u_{tx}+uu_{xx})(t,q(t,x_0(t))).$$
Substituting $(t,q(t,x_0(t)))$ into the equation (\ref{e4.1}), we obtain
\begin{align*}
\dfrac{d}{d t}w(t)=-w^2(t)+3\mu_0u(t,q(t,x_0(t)))-3\mu_0^2.
\end{align*}
By the assumption of the theorem, $\mu_0u(t,q(t,x_0(t))) = \mu_0 u(t, \xi(t)) \le0$. This implies that
\begin{align*}
\dfrac{d}{d t}w(t) \le -w^2(t)-3\mu_0^2,\quad t\ge0.
\end{align*}
Let $t_1=(1+|u'_0(\xi_0)|)/(3\mu^2_0)$. Then similar to the
proof of  Theorem \ref{t4.3}, one finds that $w(t_1)\le-1$ and
$$-\dfrac1{w(t)}-1\le-\dfrac1{w(t)}+\dfrac1{w(t_1)}\le-(t-t_1),\qquad t\in[t_1,T).$$
Therefore
 $$w(t)\le\dfrac1{t-t_1-1}\rightarrow-\infty, \quad\text{as}\quad t\rightarrow t_1+1,$$
which implies that  $ \displaystyle T\le t_1+1<+\infty $ with
 $$\liminf\limits_{t\uparrow T}\left(\inf\limits_{x\in\mathbb{S}}u_x(t,x)\right)=-\infty.$$
 This completes the proof of Theorem \ref{t4.7}.
\end{proof}

\begin{remark}It is observed in the proof of Theorem \ref{t4.7} that we may control the sign of the
conserved quantity $\tilde{H}_2$ to guarantee the blow-up  solutions
of the $\mu$DP equation. This  new method is expected to have
further applications to other nonlinear wave equations with a part of the Burgers equation.
\end{remark}

We are now in a position to give another interesting blow-up result for the
$\mu$DP equation.
\begin{theorem}\label{t4.4}
Let $u_0 \in H^s(\mathbb{S}), s > 3/2$ and $T>0$ be the maximal time of existence of the corresponding
solution $u(t,x)$ to (\ref{e1.2}) with the initial data $u_0$. If
$$|\mu_0|<\sqrt{\frac{32\mpi^2-9}{32\mpi^2}}\mu_2,
$$
where $\mu_0,\mu_2$ are defined in
(\ref{2.1}) and (\ref{2.3}), then the corresponding
solution $u(t,x)$ to (\ref{e1.2}) must blow up in finite time $T>0$ with
$$0<T\le\inf\limits_{\alpha \in I}\left (\dfrac{6}{4-9\alpha|\mu_0|}+2\alpha\dfrac{1+|\int_{\mathbb{S}}u_{0x}^3(x)d x|}
{72\mpi^2\alpha\mu_0^2(\mu_2^2-\mu_0^2)-9|\mu_0|\mu_2^2}\right ),$$
where $ \displaystyle I = \left (\frac{\mu_2^2}{8\mpi^2|\mu_0|(\mu_2^2-\mu_0^2)}, \, \frac{4}{9|\mu_0|} \right )$
such that
$$\liminf\limits_{t\uparrow T}\left (\inf\limits_{x\in\mathbb{S}}u_x(t,x)\right )=-\infty.$$
\end{theorem}
\begin{proof}Since the case $ \mu_0 = 0 $ was proved in \cite {len2}, we only need to show the case $ \mu_0 \neq 0. $
Similar to the proof of above theorem, it suffices to consider the
case $s=3$. Recall  (\ref{e4.1}), that is
\begin{align*}
u_{tx}+u^2_x+uu_{xx}=3u\mu_0-3\mu_0^2.
\end{align*}
Multiplying the above equation by $3u^2_x$ and integrating on
$\mathbb{S}$ with respect to $x$, we obtain
\begin{align*}
&\dfrac{d}{d t}\int_{\mathbb{S}}u_x^3\;dx=\int_{\mathbb{S}}3u_x^2u_{xt}\;dx\\
&\quad=-3\int_{\mathbb{S}}u_x^4\;dx-\int_{\mathbb{S}}3uu_x^2u_{xx}\;dx+9\mu_0\int_{\mathbb{S}}uu_x^2\;dx-
9\mu^2_0\int_{\mathbb{S}}u_x^2\;dx\\
&\quad=-2\int_{\mathbb{S}}u_x^4\;dx+9\mu_0\int_{\mathbb{S}}uu_x^2\;dx-
9\mu^2_0\int_{\mathbb{S}}u_x^2\;dx.
\end{align*}
On the other hand, it follows from Lemma \ref{l2.2} that
\begin{align*}
\int_{\mathbb{S}}(u-\mu_0)^2\;dx\le\frac1{4\mpi^2}\int_{\mathbb{S}}u_x^2\;dx.
\end{align*}
Or, what is the same,
\begin{align*}
\int_{\mathbb{S}}u_x^2\;dx\ge4\mpi^2\int_{\mathbb{S}}(u-\mu_0)^2\;dx.
 \end{align*}
It is also easy to see that
 \begin{align*}
 \int_{\mathbb{S}}(u-\mu_0)^2\;dx=\int_{\mathbb{S}}(u^2-2\mu_0u+\mu_0^2)\;dx=\mu_2^2-\mu_0^2.
 \end{align*}
This implies in turn that
\begin{align*}
\int_{\mathbb{S}}u_x^2\;dx\ge4\mpi^2(\mu_2^2-\mu_0^2).
 \end{align*}
In addition, it is observed that
$$
\int_{\mathbb{S}}uu_x^2\;dx \le \left (\int_{\mathbb{S}}u^2\;dx \right
)^{\frac1{2}} \left (\int_{\mathbb{S}}u_x^4\;dx \right )^{\frac1{2}}
 \le\frac{\alpha}{2}\int_{\mathbb{S}}u_x^4\;dx+\frac1{2\alpha}\int_{\mathbb{S}}u^2\;dx
  =\frac{\alpha}{2}\int_{\mathbb{S}}u_x^4\;dx+\frac1{2\alpha}\mu_2^2.
$$
 Therefore, combining the above inequalities yields
 $$
 \dfrac{d}{d t}\int_{\mathbb{S}}u_x^3\;dx\le \left (\frac{9}{2}\alpha|\mu_0|-2 \right )\int_{\mathbb{S}}u_x^4
 \;dx+\frac{9}{2\alpha}|\mu_0|\mu_2^2-36\mpi^2\mu_0^2(\mu_2^2-\mu_0^2).
 $$
 Set $$ \frac{9}{2}\alpha|\mu_0|-2<0,$$
which is to say, $ \displaystyle \alpha<4/(9|\mu_0|). $ If $
\alpha > 0 $ also satisfies
$$ 36\mpi^2\mu_0^2(\mu_2^2-\mu_0^2)-\frac{9}{2\alpha}|\mu_0|\mu_2^2>0,$$
then one finds that
$$\alpha>\frac{\mu_2^2}{8\mpi^2|\mu_0|(\mu_2^2-\mu_0^2)}.$$
By the assumption of this theorem, we know that
$$|\mu_0| < \sqrt{\frac{32\mpi^2-9}{32\mpi^2}}\mu_2.$$
Therefore one can choose $\alpha>0$ such that
$$\frac{\mu_2^2}{8\mpi^2|\mu_0|(\mu_2^2-\mu_0^2)}<\alpha<\frac{4}{9|\mu_0|}.$$
Let
 $$c_1=2-\frac{9}{2}\alpha|\mu_0|>0,\quad c_2=36\mpi^2\mu_0^2(\mu_2^2-\mu_0^2)-\frac{9}{2\alpha}|\mu_0|\mu_2^2>0.
 $$
This then follows that
 $$\dfrac{d}{d t}\int_{\mathbb{S}}u_x^3\ dx\le -c_1\int_{\mathbb{S}}u_x^4\ dx-c_2\le
 -c_1\left (\int_{\mathbb{S}}u_x^3\ dx\right )^{\frac{4}{3}}-c_2.$$
 Again, define  $V(t)=\int_{\mathbb{S}}u_x^3(t, x)  \ dx$ with $ t \in [0, T). $  Then we have
 $$\dfrac{d}{d t}V(t)\le-c_1(V(t))^{\frac{4}{3}}-c_2\le-c_2<0,\quad t\in[0,T).$$
Similar to the proof of the Theorem \ref{t4.3}, we define
$t_1=(1+|V(0)|)/c_2$ and conclude $\displaystyle T \le
t_1+\frac{3}{c_1}<+\infty. $ In consequence of Theorem \ref{t4.2}, we obtain
$$
\liminf\limits_{t\uparrow T}\left
(\inf\limits_{x\in\mathbb{S}}u_x(t,x)\right )=-\infty.
$$
This completes the proof of Theorem \ref{t4.4}.
\end{proof}

\section{Blow-up rate}
Our attention is now turned to the question of the blow-up rate of
the slope to a breaking wave for the initial value problem
(\ref{e1.1}) and (\ref{e1.2}).
\begin{theorem}\label{t4.9}
Let $u(t,x)$  be the solution to the initial value problem
(\ref{e1.1}) associated with the $\mu$CH equation with initial data
$u_0 \in H^s(\mathbb{S}), s > 3/2 $. Let $T>0$ be the maximal time of existence
of the solution $u(t,x)$. If $T<\infty$, we have
$$\underset{t\uparrow T}{\lim}\left\{\underset{x\in \mathbb{S}}{\inf}
[u_x(t,x)(T-t)]\right\}=-2$$ while the solution remains
uniformly bounded.
\end{theorem}
\begin{proof}
The uniform boundedness of the solution can be easily obtained by
the priori estimate (\ref{e2.2}). By Lemma \ref{l2.3}, we can see
that the function $$w(t) = \inf\limits_{x\in\mathbb{S}}u_x(t, x)=u_x(t,\xi(t))$$ is locally Lipschitz with
$w(t)<0,\  t\in[0,T)$. Note that $u_{xx}(t,\xi(t))=0$ for a.e. $t\in(0,T)$.

It follows from  Theorem \ref{t4.5} that
\begin{equation}\label{4.9}
\frac {d}{dt} w(t)\le-\frac1{2}w^2(t)+N,\quad t\in[0,T),
\end{equation}
where $N=\frac1{2}K^2$ and  $\displaystyle K  = \sqrt{2\mu_1 \left ( \Big | \frac{\sqrt{3}}{3}  |\mu_0|-\frac1{2}\mu_1 \Big | \right )}. $

Now fix any $\varepsilon\in(0,1/2)$. From Theorem \ref{t4.2},
there exists $t_0\in(0,T)$ such that
$w(t_0)<-\sqrt{2N+\frac{N}{\varepsilon}}$. Notice that $w(t)$ is
locally Lipschitz so that it is absolutely continuous on $[0,T)$. It
then follows from the above inequality that $w(t)$ is decreasing on
$[t_0,T)$ and satisfies that
$$w(t)<-\sqrt{2N+\frac{N}{\varepsilon}}<-\sqrt{\frac{N}{\varepsilon}}, \  t\in[t_0,T).$$
Since $w(t)$ is decreasing on $[t_0,T)$, it follows that
$$\lim\limits_{t\uparrow T}w(t)=-\infty.$$
From (\ref{4.9}), we obtain
$$\frac1{2}-\varepsilon\le\frac{d}{d t}\left(w(t)^{-1}\right)=-\frac{w'(t)}{w^2(t)}\le\frac1{2}+\varepsilon.$$
Integrating the above equation on $(t,T)$ with $t\in(t_0,T)$ and
noticing that $\lim\limits_{t\uparrow T}w(t)=-\infty$, we get
$$ \left (\frac1{2}-\varepsilon \right )(T-t)\le-\frac1{w(t)}\le \left (\frac1{2}+\varepsilon \right )(T-t).$$
Since $\varepsilon\in \left ( 0,\frac1{2} \right )$ is arbitrary, in
view of the definition of $w(t)$, the above inequality implies the
desired result of Theorem \ref{t4.9}.
\end{proof}

\begin{theorem}\label{t4.10}
Let $u(t,x)$  be the solution to the initial value problem (\ref{e1.2}) associated with the $\mu$DP equation with initial data $u_0 \in
H^s(\mathbb{S}), s > 3/2 $. Let $T>0$ be the maximal time of existence of the solution $u(t,x)$. If
$T<\infty$, we have
$$\underset{t\uparrow T}{\lim}\left\{\underset{x\in \mathbb{S}}{\inf}
[u_x(t,x)(T-t)]\right\}=-1$$ while the solution remains
uniformly bounded.
\end{theorem}
\begin{proof}
Again we may assume $s=3$ to prove the this theorem. The uniform boundedness of the solution can be easily obtained by
the priori estimate in Lemma \ref{l2.7}. By (\ref{e4.1}), we have
\begin{align*}
u_{tx}+u^2_x+uu_{xx}=3u\mu_0-3\mu_0^2.
\end{align*}

It is inferred from  Lemma \ref{l2.3} that the function $$w(t) = \inf\limits_{x\in\mathbb{S}}u_x(t, x)=u_x(t,\xi(t))$$
is locally Lipschitz with $w(t)<0,\  t\in[0,T)$. Note that $u_{xx}(t,\xi(t))=0$ for a.e. $t\in(0,T)$. Then we deduce that
\begin{equation}\label{4.11}
w'(t)=- w^2(t)+3u(t,\xi(t))\mu_0-3\mu_0^2,\quad t\in[0,T),
\end{equation}
It follows from Lemma \ref{l2.7} that
\begin{align*}
|3u(t,\xi(t))\mu_0-3\mu_0^2|&\le3|\mu_0|\|u(t,x)\|_{L^{\infty}}+3\mu_0^2\\
&\le3|\mu_0|\left(\frac{3}{2}\mu_0^2+6|\mu_0|\mu_2\right)T+3|\mu_0|\|u_0\|_{L^{\infty}}+3\mu_0^2, \quad t\in[0,T].
\end{align*}
Set
$$N(T)=3|\mu_0|\left(\frac{3}{2}\mu_0^2+6|\mu_0|\mu_2\right)T+3|\mu_0|\|u_0\|_{L^{\infty}}+3\mu_0^2.$$
Combining above estimates, we deduce that
\begin{equation}\label{4.12}
w'(t)\le-w^2(t)+N(T),\quad t\in[0,T),
\end{equation}

Now fix any $\varepsilon\in(0,1)$. In view of  Theorem \ref{t4.2},
there exists $t_0\in(0,T)$ such that
$w(t_0)<-\sqrt{N(T)+\frac{N(T)}{\varepsilon}}$. Notice that $w(t)$ is
locally Lipschitz so that it is absolutely continuous on $[0,T)$. It then
follows  from the above inequality that $w(t)$ is decreasing on
$[t_0,T)$ and satisfies that
$$w(t)<-\sqrt{N(T)+\frac{N(T)}{\varepsilon}}<-\sqrt{\frac{N(T)}{\varepsilon}}, \  t\in[t_0,T).$$
Since $w(t)$ is decreasing on $[t_0,T)$, it follows that
$$\lim\limits_{t\uparrow T}w(t)=-\infty.$$
It is found from  (\ref{4.12}) that
$$1-\varepsilon\le\frac{d}{d t}(w(t)^{-1})=-\frac{w'(t)}{w^2(t)}\le1+\varepsilon.$$
Integrating the above equation on $(t,T)$ with $t\in(t_0,T)$ and
noticing that $\lim\limits_{t\uparrow T}w(t)=-\infty$, we get
$$ \left (1-\varepsilon \right )(T-t)\le-\frac1{w(t)}\le \left (1+\varepsilon \right )(T-t).$$
Since $\varepsilon\in \left ( 0,1\right )$ is arbitrary, in
view of the definition of $w(t)$, the above inequality implies the
desired result of Theorem \ref{t4.10}.
\end{proof}

\section*{Appendix A. Geometric descriptions of these two equations}

Integrable equations solved by the inverse scattering transformation
method have elegant geometric interpretations. Several different
geometric frameworks have been used to provide geometric
interpretations to integrable systems. Besides the Arnold's approach
to Euler equations on Lie groups \cite{arn} (see also more recent
exposition \cite{khe2}), other two geometric descriptions are
important in the study of integrable systems. In an intriguing work
due to Chern and Tenenblat \cite{che1}, they provide a
classification of a class of nonlinear evolution equations
describing pseudo-spherical surfaces. As a consequence, many
integrable equations are shown to be geometrically integrable. The
approach also provides a direct approach to compute conservation
laws of integrable systems. Another interesting geometric
interpretation is provided through invariant geometric curve or
surface flows. For instances, the mKdV equation, the Schr\"{o}dinger
equation, the KdV equation and the Sawada-Kotera equation arise
naturally from non-stretching invariant curve flows in Klein
geometries (see \cite{gol,has,cho1,cho2} and references therein). It
is noticed that the celebrated CH equation arises from a
non-stretching invariant plane curve flow in centro-equiaffine
geometry \cite{cho1} and describes pseudo-spherical surfaces
\cite{rey}.

In this Appendix, we show that the $\mu$CH equation and $\mu$DP
equation also arise from non-stretching invariant curve flows
respectively in plane centro-equiaffine geometry $CA^2$ and
three-dimensional centro-equiaffine geometry $CA^3$. Furthermore, we
show that the $\mu$CH equation describes pseudo-spherical surfaces,
and the $\mu$DP equation describes affine surfaces.

First, we study non-stretching invariant plane curve flows in
centro-equiaffine geometry $CA^2$. Consider a star-shaped plane
curve $\gamma(p)$ with a parameter $p$, i.e., the curve satisfies
$[\gamma, \gamma_p]\not=0$, here $[\gamma_1,\gamma_2]$ denotes the
determinant of two vectors $\gamma_1$ and $\gamma_2$. Along the
curve one may represent it by a special parameter $\sigma$
satisfying
\begin{equation*} \tag{A.1} [\gamma,\gamma_{\sigma}]=1.
\end{equation*} In terms of the free parameter $p$, the centro-equiaffine
arc-length is given by
\begin{equation*} d\sigma=[\gamma,\gamma_p]dp.\end{equation*} It follows from
(A.1) that there exists a function $\phi$ such that
\begin{eqnarray*} \gamma_{\sigma\sigma}+\phi \gamma=0,\end{eqnarray*}
where\begin{eqnarray*} \phi=[\gamma_{\sigma},\gamma_{\sigma\sigma}]
\end{eqnarray*} is centro-equiaffine curvature of $\gamma(\sigma)$. It is easily to
verify that $\phi$ is invariant with respect to the linear
transformation \begin{eqnarray*} \gamma'=A\gamma, \quad A\in
SL(2,{\Bbb R}). \end{eqnarray*} The centro-equiaffine tangent and
normal vectors of $\gamma$ are defined to be \begin{eqnarray*} {\bf
T}=\gamma_{\sigma},\quad {\bf N}=-\gamma. \end{eqnarray*} Hence we
have the Serret-Frenet formulae for $\gamma$
\begin{equation*} {\bf T}_{\sigma}=\phi {\bf N},\qquad {\bf
N}_{\sigma}=-T.\end{equation*}

Consider the plane curve flow for $\gamma(\sigma, t)$ in $CA^2$
\begin{equation*} \gamma_t=f{\bf N}+g{\bf T}, \tag{A.2} \end{equation*} where
$f$ and $g$ are respectively the normal and tangent velocities,
which depend on the centro-equiaffine curvature $\phi$ and its
derivatives with respect to arc-length $\sigma$.

Let $d\sigma=\xi dp$, and $L=\oint d\sigma$ be the centro-equiaffine
perimeter for a closed curve. Assume that the centro-equiaffine
arc-length does not depend on time and $L$ is invariant along the
flow, then the velocities $f$ and $g$ satisfy \begin{equation*}
\tag{A.3}
\begin{aligned}
g_{\sigma}-2f=0,\qquad \oint_{\gamma} f\ d\sigma=0.
\end{aligned}
\end{equation*}
A direct computation shows that the curvature satisfies the equation
\begin{equation*}\tag{A.4}
\phi_t=(D_{\sigma}^2+4\phi+2\phi_{\sigma}\partial_{\sigma}^{-1})f
\end{equation*}
after using (A.3).

Setting $\phi\equiv m=\delta u-u_{\sigma\sigma}$ and
$f=-u_{\sigma}$, we obtain the CH and HS equations
\begin{eqnarray*} m_t+u_{\sigma\sigma\sigma}+4m
u_{\sigma}+2u m_{\sigma}=0,\end{eqnarray*} respectively for
$\delta=1$ and $\delta=0$. After the change of variables
\begin{equation*} \tag{A.5} t\rightarrow t,\;\;\sigma \rightarrow
x-t,\;\;u\rightarrow \frac 12 u,\end{equation*}  we get the CH and
HS equations in the standard form
\begin{equation*}\tag{A.6}
m_t+2m u_x+u m_x=0,\end{equation*} where $m=\delta u-u_{xx}$.

Now we assume that $\phi$ is periodic, i.e.,
$\phi(\sigma+1)=\phi(\sigma)$. Let
$\phi=\int_0^1u\;d\sigma+h_{\sigma}$, where $h(t,\sigma)$ is also a
periodic function of $\sigma$. It implies that
\begin{eqnarray*} \int_0^1u\;d\sigma=\int_0^1\phi\;
d\sigma=\mu(u).\end{eqnarray*} Namely, $\mu(u)$ is the mean
curvature of $\gamma$. Taking \begin{eqnarray*} h=-u_{\sigma},\quad
f=-u_{\sigma}.\end{eqnarray*} We arrive at the equation
\begin{eqnarray*}
m_t+u_{\sigma\sigma\sigma}+4mu_{\sigma}+2um_{\sigma}=0.\end{eqnarray*}
The change of variables (A.5) leads to the $\mu$CH equation
\begin{equation*}\tag{A.7} m_t+2mu_{x}+um_{x}=0,\;\;
m=\mu(u)-u_{xx}.\end{equation*}

Here we make a remark about the geometric description of the
$\mu$CH equation. \\

\noindent {\bf Remark A.1.} {\em It was shown by Reyes \cite{rey}
that CH and HS equations describe pseudo-spherical surfaces.
Similarly, we can show that the $\mu$CH equation also describes
pseudo-spherical surfaces, i.e., there exist one-forms
\begin{eqnarray*}\begin{aligned}
\omega_1&=\frac 12 (\lambda m -\frac 12 \lambda^2+2)\ dx+\frac 12
[\frac 12 \lambda^2 u-\lambda( u_x+um+\frac 12)+\mu(u)-2u +\frac
2{\lambda}]\;dt,\\
\omega_2&=\lambda \ dx +(1- \lambda  u+u_x)\;dt,\\
\omega_3&=\frac 12 (\lambda m -\frac 12 \lambda^2-2)\ dx+\frac 12
[\frac 12 \lambda^2 u-\lambda( u_x+um+\frac 12)+\mu(u)+2u -\frac
2{\lambda}]\;dt,
\end{aligned}\end{eqnarray*}
which satisfy the structure equations for pseudo-spherical surface
\begin{eqnarray*}
 d\omega_1=\omega_3 \wedge \omega_2,\quad
d\omega_1=\omega_3 \wedge  \omega_2,\quad d\omega_1=\omega_3 \wedge
\omega_2.\end{eqnarray*}}

Based on the structure equations, using the equations for
pseudo-potential, we are able to obtain two sets of conservation
laws of the $\mu$CH equation.

Next, we consider non-stretching invariant space curve flows in
three-dimensional centro-equiaffine geometry $CA^3$. The isometries
of the centro-equiaffine geometry consists of the linear
transformations $x'=Ax$, $A\in SL(3,{\Bbb R})$. For a general curve
$\gamma(p)$ with a parameter $p$, satisfying $[\gamma,\gamma_p,
\gamma_{pp}]\not=0$, along the curve one may reparametrize it by a
special parameter $\sigma$ satisfying
\begin{equation*}\tag{A.8}
[\gamma,\gamma_{\sigma},\gamma_{\sigma\sigma}]=1,
\end{equation*}
everywhere, when $[\gamma_1,\gamma_2,\gamma_3]$ denotes the
determinant of the vectors $\gamma_1$, $\gamma_2$ and $\gamma_3$. In
terms of an free parameter $p$, the centro-equiaffine arc-length
$\sigma$ is defined to be
\begin{equation*}
d\sigma=[\gamma,\gamma_p,\gamma_{pp}]^{\frac 13 }dp.
\end{equation*}
It follows from (A.8) that there exist two functions $\alpha$ and
$\beta$ such that
\begin{equation*}
\gamma_{\sigma\sigma\sigma}=\alpha\gamma+\beta\gamma_{\sigma},
\end{equation*}
where
\begin{eqnarray*}
\alpha=[\gamma_{\sigma},\gamma_{\sigma\sigma},\gamma_{\sigma\sigma\sigma}],\quad
\beta=[\gamma,\gamma_{\sigma\sigma\sigma},\gamma_{\sigma\sigma}].\
\end{eqnarray*}
It is readily to verify that $\alpha$ and $\beta$ are invariant with
respect to centro-equiaffine linear transformation. We define them
to be the curvatures of $\gamma$. Whence we have the Serret-Frenet
formulae
\begin{equation*}
\tag{A.9} \left(\begin{array}{c}\gamma\\
\gamma_{\sigma}\\
\gamma_{\sigma\sigma}\end{array}\right)=\left(\begin{array}{ccc}0&1&0\\
0&0&1\\
\alpha&\beta&0\end{array}\right).\end{equation*} Now consider the
invariant curve flow in $CA^3$ governed by
\begin{equation*}\label{A.10}
\gamma_t=F\gamma+G\gamma_{\sigma}+H\gamma_{\sigma\sigma}
\end{equation*}
where $F$, $G$ and $H$ are velocities, depending on the
centro-equiaffine curvatures $\alpha$ and $\beta$ and their
derivatives with respect to $\sigma$.

Assume that the arc-length $\sigma$ does not depend on time $t$,
namely $[\frac{\partial }{\partial \sigma}, \frac{\partial
}{\partial t}]=0$, which implies that the velocities satisfy
\begin{equation*}\tag{A.11}
F+G_{\sigma}+\frac 23 \beta H+\frac 13 H_{\sigma\sigma}=0.
\end{equation*}
Similar to the plane case, for a closed curve $\gamma$, one requires
that the centro-equiaffine perimeter is invariant under the curve
flow. It turns out that
\begin{equation*}\tag{A.12}
\oint(F+\frac 23 \beta H)\;d\sigma=0. \end{equation*} The evolution
for the curvatures is \begin{equation*}\tag{A.13}
\begin{aligned}
\alpha_t=&[F_{\sigma\sigma}+\alpha(G+2H_{\sigma})+\alpha_{\sigma}H]_{\sigma}+2\alpha
G_{\sigma}-\beta F_{\sigma}+\alpha H_{\sigma\sigma},\\
\beta_t=&[3F_{\sigma}+G_{\sigma\sigma}+\beta(G+2H_{\sigma})+(\alpha+\beta_{\sigma})H]_{\sigma}+2\alpha
H_{\sigma}\\&+\beta H_{\sigma\sigma}+\beta
G_{\sigma}+\alpha_{\sigma}H.
\end{aligned}
\end{equation*}

Now we consider two possibilities. First, we set $\beta=1$,
$F=u_{\sigma}+2/3$, $G=-u$, $H=-1$ and
$\alpha=-(u-u_{\sigma\sigma})$, so that (A.11), (A.12) and the
second one of (A.13) hold. Thus the first one of (A.13) becomes
\begin{equation*} m_t+3mu_{\sigma}+um_{\sigma}=0,
\tag{A.14}
\end{equation*}
which is exactly the DP equation.

Second, consider the periodic case, we choose $\beta=0$,
$F=u_{\sigma}+2/3$, $G=-u$, $H=-1$ and
$\alpha=-(\mu(u)-u_{\sigma\sigma})$. Then the second one of (A.13)
holds identically, and the first one of (A.13) becomes
\begin{equation*}\tag{A.14}
m_t+3mu_{\sigma}+um_{\sigma}=0,
\end{equation*}
where $m=\mu(u)-u_{xx}$, which is the $\mu$DP equation.

In the following, we show that $\mu$DP equation describes affine
surfaces. Let $A^3$ be the unimodular affine space of dimension
three, the unimodular affine group $G$ is generated by the following
transfomrations
\begin{eqnarray*}
\gamma'=A\gamma+B,\end{eqnarray*} where $A\in SL(3,{\Bbb R})$, $B\in
{\Bbb R}^3. $ Let $\gamma$ is the position vector of an affine
surface $M$ in $A^3$. Let $\gamma$, $e_1$, $e_2$, $e_3$ be an affine
frame on $M$ such that $e_1$, $e_2$ are tangent to $M$ at $\gamma$,
and
\begin{eqnarray*}
[e_1,e_2,e_3]=1.
\end{eqnarray*}
We write
\begin{equation*}
d\gamma=\sum\limits_j\omega^j e_j,\qquad
de_j=\sum\limits_k\omega_j^k e_k,
\end{equation*}
where $\omega^j$ and $\omega_j^k$ are the Maurer-Cartan forms of
$G$, which satisfy
\begin{equation*}\tag{A.15}
\begin{aligned}
&\sum\limits_j\omega_j^j=0,\quad
d\omega_j^l=\sum\limits_j^3\omega_j^k\omega_k^l,\\& \omega^1\wedge
\omega_1^3+\omega^2\wedge \omega_2^3=0,\quad
\omega^3=0,\\
&d\omega^2=\omega^1\wedge \omega_1^2+\omega^2\wedge \omega_2^2,\quad
d\omega^1=\omega^1\wedge \omega_1^1+\omega^2\wedge \omega_2^1,
\end{aligned}
\end{equation*}
for $j,l=1,2,3$.\\

\noindent {\bf Definition A.1.} \cite{che2} A PDE for a function
$u(t,x)$ describes affine surfaces if there exist smooth functions
$f_j^k$, $g_j^k$, $h_{pq}$, $1\leq j,k\leq 3$, $1\leq p,q\leq 2$,
depending only on $u$ and their derivatives such that the 1-forms
$\omega_j^k=f_j^k \ dx+g_j^kdt$, $\omega^p=h_{p1}\ dx+h_{p2}dt$
satisfy
the structure equations (A.15) for affine surfaces $M$.\\

We have the following theorem.\\

\noindent {\bf Theorem A.1.} The $\mu$DP equation describes affine
surfaces.\\

\begin{proof} Indeed, we take
\begin{eqnarray*}
&&f_1^1=0,\;\; g_1^1=\frac 1{2\lambda}-u_x,\;\; f_1^2=\lambda
m,\;\;g_1^2=-\frac{1}{4\lambda}(4\lambda^2 um-4\lambda u_x+1),\;\;
f_1^3=\frac{1}{2\lambda},\\
&&g_1^3=-\frac{1}{2\lambda}(2\mu(u)+u),\;\; f_2^1=0,\;\; g_2^1=\frac
1{\lambda},\;\; f_2^2=0,\;\;g_2^2=u_x-\frac{1}{2\lambda},\;\;
f_2^3=\frac 1{\lambda},\\
&&g_2^3=-\frac{u}{\lambda},\;\; f_3^1=-\lambda,\;\; g_3^1=\lambda
u,\;\; f_3^2=\frac {\lambda}2\;\;g_3^2=\lambda(\mu(u)-\frac
12u),\;\; f_3^3=0,\;\; g_3^3=0,\\
&&h_{11}=1,\;\; h_{12}=-u, \;\; h_{21}=-\frac 12 \;\;h_{22}=\frac 12
u-\mu(u).
\end{eqnarray*} A straightforward computation shows that the Maurer-Cartan forms defined by these functions satisfy the
structure equation for affine surface.
\end{proof}


\begin{thebibliography}{99}
\bibitem{arn}{\small \textsc{V. I. Arnold,}\ Sur la g\'{e}om\'{e}trie
diff\'{e}rentielle des groupes de Lie de dimenson infinite er ses
application \`{a} l'hydrodynamique des fluids parfaits, {\it Ann.
Fourier Grenoble}, {\bf 16} (1966), 319-361.}

\bibitem{bc} {\small \textsc{A. Bressan, A. Constantin,}\ Global solutions of the Hunter-Saxton
equation, {\it SIAM J. Math. Anal.}, {\bf 37} (2005), 996-1026.}

\bibitem{but} {\small \textsc{G. Buttazo, M. Giaquina, S. Hildebrandt,}\ One-Dimensional
Variational Problems: An Introduction, Clarendon Press, Oxford,
1998.}

\bibitem{cam} {\small \textsc{R. Camassa, D. D. Holm,}\ An integrable shallow water
equation with peaked solitons, {\it Phys. Rev. Lett.}, {\bf 71}
(1993), 1661-1664.}

\bibitem{che1}{\small \textsc{S. S. Chern, K. Tenenblat,}\ Pseudo-spherical surfaces and evolution equations,
{\it Stud. Appl. Math.},  {\bf 74} (1986), 55-83.}

\bibitem{che2}{\small \textsc{S. S. Chern, C.L. Terng,}\ An analogue of B\"{a}cklund
theorem in affine geometry, {\it Rocky Mountain J. Math.},  {\bf 10
} (1980), 105-124.}

\bibitem{cho1}{\small \textsc {K. S. Chou, C. Z. Qu,}\ Integrable equations arising from motions
of plane curves I, {\it Physica D},  {\bf 162} (2002), 9-33.}
\bibitem{cho2}{\small \textsc {K. S. Chou, C. Z. Qu,}\  Integrable equations arising from motions
of plane curves II, {\it J. Nonlinear Sci.},  {\bf 13} (2003), 487-517.}

\bibitem{coc} {\small \textsc{G. M. Coclite, K. H. Karlsen,}\ On the well-posdeness
of the Degasperis-Procesi equation, {\it J. Funct. Anal.}, {\bf 233}
(2006), 60-91. }

\bibitem{con1} {\small \textsc{A. Constantin,}\ On the Cauchy problem for the periodic
Camassa-Holm equation, {\it J. Differential Equations}, {\bf 141}
(1997), 218-235.}

\bibitem{con2} {\small \textsc{A. Constantin,}\ On the Blow-up of solutions of a periodic shallow
water equation, {\it J. Nonlinear Sci.}, {\bf 10} (2000), 391-399.}

\bibitem{con3} {\small \textsc{A. Constantin, J. Escher,}\ Well-posedness, global
existence and blow-up phenomena for a periodic quasi-linear
hyperbolic equation, {\it Comm. Pure Appl. Math.}, {\bf 51} (1998),
475-504.}

\bibitem{con4} {\small \textsc{A. Constantin, J. Escher,}\ On the blow-up rate and
the blow-up set of breaking waves for a shallow water equation, {\it
Math. Z.}, {\bf 233} (2000), 75-91.}

\bibitem{con5} {\small \textsc{A. Constantin, J. Escher,}\ Wave breaking for nonlinear
nonlocal shallow water equations, {\it Acta Math.}, {\bf 181}
(1998), 229-243.}

\bibitem{conl2}
{\small \textsc{A. Constantin, D. Lannes,}\ The hydrodynamical
relevance  of the Camassa-Holm and Degasperis-Procesi equations,
{\it Arch. Ration. Mech. Anal.,} {\bf 192} (2009), 165-186.}

\bibitem{con6} {\small \textsc{A. Constantin, H. P. McKean,}\  A shallow water equation on
 the circle, {\it Comm. Pure Appl. Math.}, {\bf 52} (1999), 949-982.}

 \bibitem{deg1} {\small \textsc{A. Degasperis, M. Procesi, }\  Asymptotic integrability, {\it Symmetry
and perturbation theory (Rome, 1998),}  {\bf 23} (World Sci. Publ., River
Edge, NJ, 1999)}.

\bibitem{deg2} {\small \textsc{A. Degasperis, D. D. Holm, A. N. W. Hone, }\ Integrable
and non-integrable equations with peakons, Nonlinear physics:
theory and experiment, \textbf{II} (Gallipoli, 2002), 37 (World Sci.
Publ., River Edge, NJ, 2003)}.

\bibitem{esc1} {\small \textsc{J. Escher, B. Kolev,}\ The Degasperis-Procesi equation as a non-metric
Euler equation, arXiv:0908.0508v1, preprint.}

\bibitem{esc2} {\small \textsc{J. Escher, M. Kohlmann, B. Kolev,}\ Geometric aspects of the periodic $\mu$DP equation,
arXiv:1004.0978v1, preprint.}

\bibitem{esc3} {\small \textsc{J. Escher, Y. Liu, Z. Yin,}\ Global weak solutions and blow-up
structure for the Degasperis-Procesi equation, {\it J. Funct. Anal.},
{\bf 241} (2006), 457-485.}

\bibitem{esc4} {\small \textsc{J. Escher, Y. Liu, Z. Yin,}\ Shock waves and blow-up
phenomena for the periodic Degasperis-Procesi equation, {\it Indiana
Univ. Math. J.}, {\bf 56} (2007), 87-117.}

\bibitem{fuc} {\small \textsc{B. Fuchssteiner, A. S. Fokas,}\ Symplectic structures,
their B\"{a}cklund transformations and hereditary symmetries, {\it
Physica D}, {\bf 4} (1981/1982), 47-66.}

\bibitem{gol}{\small \textsc{R. E. Goldstein, D. M. Petrich,}\ The Korteweg-de Vries
hierarchy as dynamics of closed curves in the plane, {\it Phys. Rev.
Lett.}, {\bf 67} (1991), 3203-3206.}

\bibitem{has}{\small \textsc{H. Hasimoto,}\ A soliton on a vortex filament, {\it J. Fluid Mech.},  {\bf 51}
(1972), 477-485.}

\bibitem{hun1} {\small \textsc{J. K. Hunter, R. Saxton,}\
Dynamics of director fields, {\it SIAM J. Appl. Math.}, {\bf 51}
(1991), 1498-1521.}

\bibitem{hun2} {\small \textsc{J. K. Hunter, Y. Zheng,}\ On a completely integrable hyperbolic
variational equation, {\it Physica D}, {\bf 79} (1994), 361-386.}

\bibitem{kat1}  {\small \textsc{T. Kato,}\ On the Korteweg-de Vries equation, {\it Manuscripta Math.},
{\bf 28} (1979), 89-99.}


\bibitem{kat}  {\small \textsc{T. Kato, G. Ponce,}\ Communtator estimates and the
Euler and Navier-Stokes equations, {\it Comm. Pure Appl. Math.},
{\bf 41} (1988), 891-907.}

\bibitem{khe} {\small \textsc{B. Khesin, J. Lenells, G. Misiolek,}\ Generalized Hunter-Saxton
equation and the geometry of the group of circle diffeomorphisms, {\it Math. Ann.}, {\bf 342}
(2008), 617-656.}

\bibitem{khe2} {\small \textsc{B. Khesin, G. Misiolek,}\ Euler equations on
homogeneous spaces and Virasoro orbits, {\it Adv. Math.},  {\bf 176}
(2003), 116-144.}

\bibitem{kou}{\small \textsc{S. Kouranbaeva,}\ The Camassa-Holm equation as a geodesic
flow on the diffeomorphism group, {\it J. Math. Phys.},  {\bf 40}
(1999), 857-868.}

\bibitem{len1}{\small \textsc{ J. Lenells,}\ The
Hunter-Saxton equation describes the geodesic flow on a sphere, {\it
J. Geom. Phys.},  {\bf 57} (2007), 2049-2064.}

\bibitem{len2} {\small \textsc{J. Lenells, G. Misiolek, F. Ti\u{g}lay,}\ Integrable evolution equations on
spaces of tensor densities and their peakon solutions, {\it Comm.
Math. Phys.}, {\bf 299}  (2010), 129-161.}

\bibitem{liu} {\small \textsc{Y. Liu, Z. Yin,}\ Global existence and blow-up phenomena
for the Degasperis-Procesi equation, {\it Comm. Math. Phys.}, {\bf
267} (2006), 801-820.}

\bibitem{lu} {\small \textsc{H. Lundmark,}\ Formation and dynamics of shock waves in the Degasperis-Procesi equation,
{\it J. Nonlinear Sci.}, {\bf 17} (2007), 169-198.}

\bibitem{mis1} {\small \textsc{G. Misiolek,}\ Classical solutions of the periodic Camassa-Holm
equation, {\it Geom. Funct. Anal.}, {\bf 12} (2002), 1080-1104.}

\bibitem{mis2} {\small \textsc{G. Misiolek,}\ A shallow water equation as a geodesic
flow on the Bott-Virasoro group, {\it J. Geom. Phys.}, {\bf 24} (1998), 203-208.}

\bibitem{rey} {\small \textsc{E. G. Reyes,}\ Geometric integrability of the Camassa-Holm equation, {\it Lett. Math. Phys.},
{\bf 59} (2002), 17-131.}

\bibitem{wh} {\small \textsc{G. B. Whitham,}\ Linear and Nolinear
Waves, John Wiley \& Sons, New York, 1974.}

\bibitem{yin2} {\small \textsc{Z. Yin,}\  On the Cauchy problem for an integrable equation with peakon solutions, {\it Illinois J. Math.}, {\bf 47} (2003), 649-666.}

\bibitem{yin3} {\small \textsc{Z. Yin,}\  On the structure of solutions to the periodic Hunter-Saxton equation, {\it SIAM J. Math. Anal.}, {\bf 36} (2004), 272-283.}


\end{thebibliography}
\end{document}